\documentclass[11pt,a4paper,twoside]{amsart}

\usepackage{amsfonts, amsmath,amssymb}
\usepackage[lite]{amsrefs}
\usepackage{amsthm}
\usepackage{tikz}
\usetikzlibrary{matrix,arrows}
\usepackage{wrapfig}
\usepackage{epsf}
\usepackage{graphicx}
\usepackage{color}

\usepackage{hyperref}
\usepackage[all]{xy}
\usepackage{geometry}
\usepackage{amscd}
\usepackage{epstopdf}

             \newcommand{\edgegraph}{\includegraphics[width=4.2mm, height=2mm]{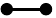}}
             \newcommand{\circlegraph}{\includegraphics[width=4.2mm, height=3.1mm]{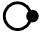}}
             \newcommand{\fundamentalDomain}{\includegraphics[width=37mm]{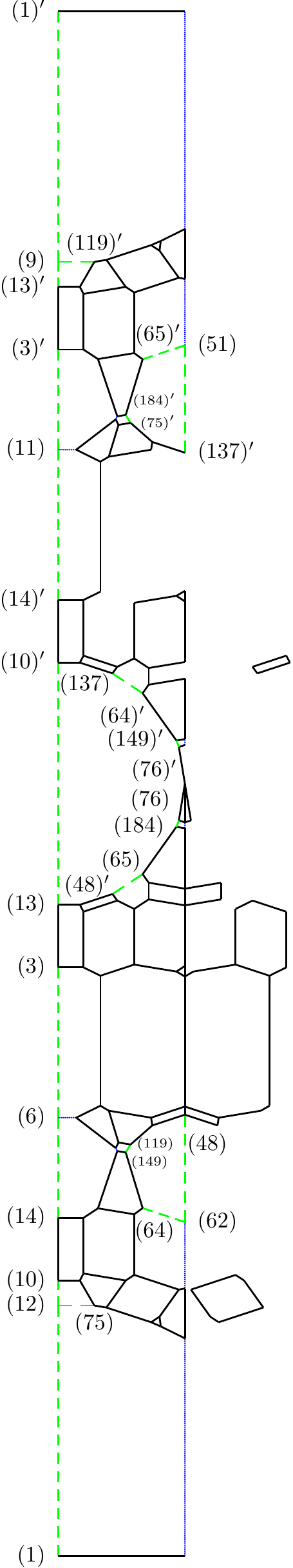}}
             \newcommand{\graphTwo}{\includegraphics[width=6mm, height=3.1mm]{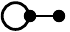}}
             \newcommand{\graphFive}{\includegraphics[width=4.2mm, height=3.1mm]{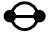}}

\newcommand{\xsp}{X_s^\prime}
\newcommand{\hyper}{\mathcal{H}}
 \newcommand{\SLO}{\mathrm{SL}_2\left(\mathcal{O}_{-m}\right)}
 \newcommand{\PSLO}{\mathrm{PSL}_2\left(\mathcal{O}_{-m}\right)}
 \newcommand{\leftSL}{\mathrm{SL}_2\left(\mathcal{O}_}
 \newcommand{\rightSL}{\right)}
 \newcommand{\SLtwo}{\mathrm{SL}_2}
 \newcommand{\PSLtwo}{\mathrm{PSL}_2}
\newcommand{\Z}{{\mathbb{Z}}}
\newcommand{\calO}{{\mathcal O}}
\newcommand{\ringOm}{\mathcal{O}_{-m}}
\newcommand{\rationals}{{\mathbb{Q}}}
\newcommand{\Center}{{Z(\Gamma)}}
\newcommand{\F}{{\mathbb{F}}}

\newcommand*{\Homol}{\operatorname{H}}
\newcommand*{\Cohomol}{\operatorname{H}}
\newcommand{\image}{{\rm image \thinspace}}

\newcommand{\sign}{{\rm sign}}
\newcommand{\Di}{{\bf Di}}
\newcommand{\Q}{{\bf Q}_8}
\newcommand{\Te}{{\bf Te}}
\newcommand{\sphere}{\mathbb{S}}

\theoremstyle{plain}
\newtheorem{thm}{\bfseries Theorem}
\newtheorem{theorem}[thm]{\bfseries Theorem}
\newtheorem{lemma}[thm]{\bfseries Lemma}
\newtheorem{proposition}[thm]{\bfseries Proposition}
\theoremstyle{remark}
\newtheorem{note}[thm]{\bfseries Note}
\newtheorem*{caveat}{\bfseries Caveat}

\newtheorem{df}[thm]{\bfseries Definition}

\newtheorem*{ConditionA}{\bfseries Condition A}
\newtheorem*{ConditionB}{\bfseries Condition B}
\newcommand{\cellCondition}{A}
\newcommand{\weakerCondition}{B}
\newtheorem{observation}[thm]{\bfseries Observation}

\theoremstyle{plain}

\newtheorem{prop}[thm]{Proposition}

\newtheorem{corollary}[thm]{Corollary}

\theoremstyle{definition}
\newtheorem{lemo}[thm]{($o$) Lemma}
\newtheorem{lemiota}[thm]{($\iota$) Lemma}
\newtheorem{lemtheta}[thm]{($\theta$) Lemma}
\newtheorem{lemthetarho}[thm]{($\theta$, $\rho$) Lemma}
\newtheorem{lemrho}[thm]{($\rho$) Lemma}
\newtheorem{defn}[thm]{Definition}
\newtheorem{rem}[thm]{Remark}

\newtheorem{notation}[thm]{Notation}

\newtheorem{remark}[thm]{\bfseries Remark}

\newcommand{\ef}{{\mathbb F}_2}
\newcommand{\Dthree}{{\bf D}_3}
\newcommand{\Dtwo}{{\bf D}_2}
\newcommand{\Af}{{\bf A}_4}
\newcommand*{\Homp}[1]{\operatorname{H}^{#1}_{{\rm red}}}
\newcommand*{\Home}[1]{\operatorname{H}^{#1}_{{\rm nil}}}

\newcommand{\ararrow}{\buildrel \alpha \over \rightarrow}
\newcommand{\brarrow}{\buildrel \beta \over \rightarrow}
\newcommand{\drarrow}{\buildrel \delta \over \rightarrow}

\DeclareMathOperator{\im}{im}

\DeclareMathOperator{\rank}{rank}

\begin{document}

\title[The mod 2 cohomology rings of $\SLtwo$ of the imaginary quadratic integers]{The mod 2 cohomology rings of \\ $\SLtwo$ of the imaginary quadratic integers}
\author[Ethan Berkove and Alexander D. Rahm]{Ethan Berkove and Alexander D. Rahm\\
(with an appendix by Aurel Page)}
\address{ Mathematics Department, Lafayette College,  Quad Drive, Easton, PA 18042, USA }
\email{berkovee@lafayette.edu}
\urladdr{http://sites.lafayette.edu/berkovee/}
\address{National University of Ireland, Department of Mathematics, University Road, Galway, Ireland}
\email{Alexander.Rahm@nuigalway.ie}
\urladdr{http://www.maths.nuigalway.ie/~rahm/}
\date{\today}

\begin{abstract}
We establish general dimension formulae for the second page of the equivariant spectral sequence of the action of the $\SLtwo$ 
groups over imaginary quadratic integers on their associated symmetric space.  
On the way, we extend the torsion subcomplex reduction technique to cases where the kernel of the group action is nontrivial. 
Using the equivariant and Lyndon--Hochschild--Serre spectral sequences, 
we investigate the second page differentials and show how to obtain the
mod~$2$ cohomology rings of our groups from this information.


\end{abstract}

\maketitle
\section{Introduction}
The objects of study in this paper are the groups $\SLtwo$ over the ring $\ringOm$
 of integers in the imaginary quadratic number field $\rationals(\sqrt{-m})$,
 with $m$ a square-free positive integer.  
These groups, as well as their central quotients $\PSLO$, are known as Bianchi groups.
The determination of the (co)homology of Bianchi groups, motivated in~\cite{Serre}, has a long history of case-by-case computations
(see \cite{higher_torsion} for a  list of references). 
This changed with the recently introduced technique of torsion subcomplex reduction, which provided general formulae 
for the Farrell cohomology of the $\PSLO$ groups for any $m$~\cite{AccessingFarrell}.
In this article, we extend the subcomplex reduction technique to the case where the kernel of a group action is nontrivial
in order to obtain the mod $2$ cohomology rings of the $\SLO$ groups.
For this purpose, we study a variant of the long exact sequence in Borel cohomology.
The reason why we consider only $\F_2$--coefficients is that for any prime $\ell > 2$,
 the Lyndon--Hochschild--Serre spectral sequence with $\F_\ell$--coefficients associated to the central extension  
\[
1 \rightarrow \{\pm 1\} \longrightarrow \SLtwo(\ringOm) \longrightarrow \PSLtwo(\ringOm)\rightarrow 1
\]
is concentrated in the horizontal axis, yielding 
$\Homol^*(\SLO;\thinspace \F_\ell) = \Homol^*(\PSLO;\thinspace \F_\ell)$.
Furthermore, Bianchi groups contain only $2$-- and $3$--torsion; results for 
$\Homol^*(\PSLO;\thinspace \F_3)$ can be found in ~\cite{AccessingFarrell}.

Philosophically, the cohomology of the Bianchi groups can be thought of as coming from two sources.
  Since the Bianchi groups act on a $2$--dimensional retract of hyperbolic $3$--space~$X$,
 the Bianchi groups have virtual (co)homological dimension $2$.
  Consequently, in dimensions $2$ and below,
 many of the cohomology classes of the Bianchi groups come from the topology of the quotient space
 and are detected with rational coefficients~\cite{SchwermerVogtmann}; 
calculations exclusively for rational coefficients have been carried out in~\cite{Vogtmann}.
 Above the virtual cohomological dimension, all cohomology classes are torsion and originate from finite subgroups.
  In particular, one can use the equivariant spectral sequence to determine the (co)homology of Bianchi groups.  
Results of the second author have made this precise.  
The article~\cite{RahmNoteAuxCRAS} 
introduced the $\ell$--torsion subcomplex, and contains a proof that this subcomplex determines the homology of $\PSLO$ above dimension~$2$; 
in addition, that the homology is a direct sum of generic modules 
associated to the homeomorphism types of connected components of the subcomplex.  

In this work, we will show how to extend the torsion subcomplex results from the projective special linear group to the linear group.  
The difficulty that we need to overcome is that the kernel of the action of $\SLO$ on $X$ is nontrivial.
We solve this problem by using a technique reminiscent of the long exact sequence in relative Borel cohomology associated to a pair of complexes  (see \cite{Henn} for a similar approach).  
This is the content of Section \ref{Section:E2 page}, where we describe the $E_2$ page of the equivariant spectral sequence in terms of components of the $2$-torsion subcomplex.  
We also need to determine whether any loops in the subcomplexes are homologous in the quotient  $_{\SLO} \backslash X$.  This interaction is tracked by the variable $c$ in our calculations.   The last piece of the puzzle in the determination of the mod $2$ cohomology of the $\SLtwo$ Bianchi groups
is the rank of the second page differential.  
We can say a lot about this rank by individually analyzing component types of a reduced $2$--torsion subcomplex 
using the cohomology ring structure over the Steenrod algebra. This is the subject of  Section~\ref{The second page differential on the components of the torsion subcomplex}.

We conclude with some sample calculations. For instance, Example $(\iota)$ in Section~\ref{examples-section} 
considers the case when $m \equiv 3 \bmod 8$,  $\rationals(\sqrt{-m})$ has precisely one finite ramification place over $\rationals$, and the ideal class number of the totally real number field $\rationals(\sqrt{m})$ is $1$.  These three conditions 
are equivalent to the quotient of the $2$--torsion subcomplex having the shape $\edgegraph$,
as worked out in~\cite{AccessingFarrell}.  Under these assumptions, our cohomology ring has the following dimensions:
$$
\dim_{\F_2}\Cohomol^{q}(\SLO; \thinspace \F_2)
=
\begin{cases}
   \beta^{1}  +\beta^{2} , &  q = 4k+5, \\
   \beta^1 +\beta^{2} +2, &  q = 4k+4, \\
   \beta^{1} + \beta^{2}+3, &  q = 4k+3, \\
   \beta^1+\beta^{2} +1, &  q = 4k+2, \\
   \beta^{1},  &  q = 1, \\
\end{cases}
$$
where $\beta^q := \dim_{\F_2}\Cohomol^q(_{\SLO} \backslash X ; \thinspace \F_2)$.  
Let $ \beta_1 := \dim_{\rationals}\Cohomol_1(_{\SLO} \backslash X ; \thinspace \rationals).$
For all absolute values of the discriminant less than $296$, numerical calculations yield
$\beta^2 +1 = \beta^1= \beta_1.$
In this range, the numbers $m$ subject to the above dimension formula and $\beta_1$ are given as follows
 (the Betti numbers are taken from~\cite{higher_torsion}).
$$\begin{array}{l|ccccccccccccccccccccccccccc}
m        &  11 & 19 & 43 & 59 & 67 & 83 & 107 & 131 & 139 & 163 & 179 & 211 & 227 & 251 & 283\\
 \hline 
\beta_1  &   1 & 1  &  2 &  4 & 3  &  5 &  6  &   8 &  7  &  7  &  10 &  10 & 12  & 14  & 13\\
\end{array}$$

We add a few remarks about the approach taken in this paper.  
The finite subgroups we consider have no subgroups isomorphic to $\Z/2 \oplus \Z/2$, 
so their cohomology is periodic of period $4$.  Furthermore, the restriction map on cohomology from any finite subgroup to the central 
$\Z/2$ subgroup is onto in dimensions divisible by $4$.  
We expect something similar for the Bianchi groups we consider, since by results of Quillen \cite{Quillen} 
their mod-$2$ cohomology is $F$-isomorphic to the cohomology of the maximal abelian subgroups, here the center $\Z/2$. 
We show this in  Proposition~\ref{prop:Bianchiperiodicity}, 
where we identify a class $\alpha$ in dimension $4$ which provides the periodicity.
We thank the referee for pointing out that this class $\alpha$ is the second Chern class of the natural representation of $\mathrm{SL}_2(\mathbb{C})$.

Under these observations, one might expect that a reasonable alternative  approach to calculating cohomology would come from using the Lyndon--Hochschild--Serre spectral sequence associated to the short exact sequence
\[
1 \rightarrow \{\pm 1\} \longrightarrow \SLtwo(\ringOm) \longrightarrow \PSLtwo(\ringOm)\rightarrow 1
\]
to determine cohomology up to degree $5$.  
Unfortunately, although one can often easily determine $d_2$ and $d_3$
for this spectral sequence, calculations in an unpublished manuscript by the first author yield that in cases with discriminant as low as 
$m = 3$ and $11$, there is a non-trivial internal $d_4$ differential whose existence can only be determined by knowing 
$\Cohomol^*( \SLtwo(\ringOm))$ beforehand.  For this reason, an approach relying exclusively on this spectral sequence does not appear feasible.

This paper grew out of an attempt to take advantage of much of what is already known about the $\PSLO$ Bianchi groups.  
Results in \cite{RahmNoteAuxCRAS}, for example, show that  information about 
$\Cohomol^*(\PSLO)$ 
above the cohomological dimension resides in ``geometric'' data. 
Specifically, for a large range of values of $m$, 
the geometric data resides in four types of subcomplexes in the quotient $_{\SLO} \backslash X$
 whose multiplicity we are going to denote by $o$, $\iota$, $\theta$, and $\rho$.  
Other calculations exist for the rational homology of $_{\SLO} \backslash X$  
(\cite{Vogtmann}, for example) which provide values of the Betti numbers $\beta_1$ and $\beta_2$ for the quotient space.  
Our work shows that these data are almost sufficient for a full answer, and give very tight bounds on the cohomology of $\SLO$ 
in all dimensions.  We note that prior machine calculations for 
$\Cohomol^*(\SLO)$ have been performed for many values of $m$, 
but the time to complete a calculation usually increases with $m$.  
Consequently, the results in this paper provide a constructive argument which can be used to complement, corroborate, and extend existing results.

\textit{Acknowledgments.} The first author thanks the De Br\'un Centre for its hospitality and for funding a stay at NUI Galway devoted to the present work.
We are grateful for a careful check by Norbert Kr\"amer,
helpful comments by Matthias Wendt on the core of our long exact sequence analysis, 
and for valuable assistance by Graham Ellis and Tuan Anh Bui on importing our cell complexes into HAP~\cite{HAP}
 for the example calculations.  Of special note, Tuan Anh Bui's resolutions for the cusp stabilizer groups have been of great help.  
We would even more like to thank the anonymous referee, who provided an approach which streamlined a number of our original arguments, and whose careful and thoughtful comments helped to  greatly improve the overall quality of this paper.  

\section{Spectral sequences and Central Extensions}\label{sec:ss}

The calculations in this paper involve two spectral sequences.  
We use the Lyndon--Hochschild--Serre spectral sequence to determine cohomology of groups via group extensions.  
We also use the equivariant spectral sequence since the Bianchi groups act cellularly 
on low-dimensional contractible complexes.  We introduce both spectral sequences briefly in this section; 
more details can be found in~\cite{AdemMilgram}, \cite{Brown}, and~\cite{McCleary}.

For the development of the Lyndon--Hochschild--Serre spectral sequence we follow the approach in~\cite{AdemMilgram}.  
Given a short exact sequence of groups 
\begin{equation}
1 \rightarrow H \rightarrow \Gamma \buildrel \pi \over \longrightarrow  \Gamma/H \rightarrow 1  \label{sss:shortexactseq},
\end{equation}
there is an associated fibration of classifying spaces.
One can then apply the Leray-Serre spectral sequence (\cite{McCleary}, Chapters 5 and 6).  This spectral sequence has  $E_2^{i,j} \cong \Homol^i(\Gamma/H; \Homol^j(H;M))$  for untwisted coefficients $M$ and converges to $\Homol^{i+j}(\Gamma;M)$.   We note that the Leray-Serre spectral sequence can also be developed more algebraically, as in~\cite{Brown}, VII.5.

In the short exact sequence of groups in Equation \ref{sss:shortexactseq}, when the normal subgroup $H$ is central in $\Gamma$, it is possible to say more.  
This is a central extension, and in such cases $\Gamma/H$ acts trivially on the homotopy fiber $B_H$. 
Then with field coefficients $M$, the $E_2$--term of the resulting spectral sequence has the form 
\[
E_2^{i,j} \cong \Homol^i(\Gamma/H; M) \otimes \Homol^j(H;M).
\]
For the remainder of this article, unless specified otherwise, we only consider cohomology with (necessarily trivial) $\F_2$-coefficients. We omit these coefficients from the notation.

Lemma IV.1.12 in~\cite{AdemMilgram} identifies a unique cohomology class $k$,  called the $k$--\textit{invariant}, which generates the kernel of $B_{\pi}^*: \Homol^2(\Gamma/H) \rightarrow \Homol^2(\Gamma)$.  The $k$--invariant is also the cohomology class associated to the extension.  The Leray-Serre spectral sequence, of which the Lyndon--Hochschild--Serre spectral sequence is 
just one example, has many useful properties.  It is compatible with cup products, for example.  

For the equivariant spectral sequence, we follow the development in~\cite{Brown}*{Chapter VII}. 
 Let the group $\Gamma$ act cellularly on a CW-space $X$ in such a way that the stabilizer of any cell fixes that cell point-wise.  
Consider the \textit{equivariant cohomology groups} $\Homol^*(\Gamma, C^\bullet(X))$  
with coefficients in the cellular co-chain complex $C^\bullet(X)$; in our setting, our co-chains will take on values in $\ef$.  
One can define these cohomology groups by taking a projective $\ef [\Gamma]$-resolution, $F$, of $\ef$,
 then setting $\Homol^*(\Gamma, C^\bullet(X)) = \Homol^*({\rm Hom}_{\F_2[\Gamma]} (F,C^\bullet(X)))$.   When $X$ is a contractible space, 
the equivariant cohomology groups can be identified with the
cohomology of $\Gamma$, as
\[ 
\Homol^*(\Gamma, C^\bullet(X)) \cong \Homol^*(\Gamma, C^\bullet({\rm point})) \cong \Homol^*(\Gamma).
\]
Using the horizontal and vertical filtrations of the double complex ${\rm Hom}_{\Gamma} (F,C^\bullet(X))$ we get a spectral sequence with 
\[
E_1^{i,j} \cong \Homol^j(\Gamma, C^i(X))
\]
 converging to $\Cohomol^{i+j}(\Gamma)$.     Let $X^{(i)}$ be a set of representatives of $i$--cells in $X$
 and let $\Gamma_\sigma$ be the stabilizer in $\Gamma$ of a cell $\sigma$.  Then 
\[
C^i(X)   = \prod\limits_{\sigma \in X^{(i)}} \ef \cong \prod \limits_{\sigma \thinspace\in \thinspace _\Gamma \backslash X^{(i)}} {\rm Coind}^\Gamma_{\Gamma_\sigma} \ef .
\]
Once we apply Shapiro's Lemma and sum over the representatives of $i$--cells in $\Gamma \backslash X_i$, the spectral sequence takes the form
$
E_1^{i,j} \cong \prod\limits_{\sigma \thinspace\in \thinspace _\Gamma \backslash X^{(i)}} \Homol^j(\Gamma_\sigma).
$

The equivariant  spectral sequence has a number of additional desirable properties that we will use in the sequel.  
\begin{enumerate}
\item\label{ss1} There is a product on the spectral sequence, $E_r^{pq}  \otimes E_r^{st} \rightarrow E_r^{p+s,q+t}$, which is compatible with the standard cup product on $\Cohomol^*(\Gamma)$  \cite{Brown}*{VII.5}
\item\label{ss2} On the $E_2$--page,
the products in  $E_2^{0,*}$, the vertical edge, are compatible with the products in ${\prod}_ {\sigma \in _\Gamma \backslash X^{(0)} }   \Cohomol^q(\Gamma_{\sigma})$.  \cite{Brown}*{X.4.5.vi}.
\item\label{ss3} The differential $d_1$ is a difference of restriction maps (cohomology 
analog of~\cite{Brown}*{VII.8})
\[
\prod_{\sigma \thinspace\in \thinspace _\Gamma\backslash X^{(i)}} \Homol^j (\Gamma_\sigma)
\xrightarrow{\ d_1^{i,j}}\ 
\prod_{\tau \thinspace\in \thinspace _\Gamma\backslash X^{(i+1)}} \Homol^j(\Gamma_\tau).
\]
\end{enumerate}



\begin{rem}\label{rem:derivation}
In both the equivariant and Lyndon--Hochschild--Serre spectral sequences, the differentials are derivations.  Working mod $2$, this means that given cohomology classes $u, v \in E_r$, the differential satisfies the Leibniz rule, 
\[
d_r(u v) = d_r(u) v +  u d_r(v).
\] 

A proof that the differential is a derivation for the Leray Serre spectral sequence can be found in \cite{McCleary} (Section 1.4 and Theorem 2.14).  For the equivariant spectral sequence, we lay out a sketch following Brown \cite{Brown}*{X.4}.   Recall that $\Homol^*(\Gamma, C^\bullet(X)) = \Homol^*({\rm Hom}_{\Gamma} (F,C^\bullet(X)))$, 
and note that  $F \otimes F$ is also a projective $\ef [\Gamma]$-resolution of $\ef$.  Therefore, there is a diagonal approximation $\triangle: F \rightarrow F \otimes F$.  Consider the composition
\begin{equation}
\begin{split}
{\rm Hom}_{\Gamma} (F,C^\bullet(X)) \otimes  {\rm Hom}_{\Gamma} (F,C^\bullet(X)) & \longrightarrow  {\rm Hom}_{\Gamma} (F \otimes F,C^\bullet(X) \otimes C^\bullet(X)) \\		
										 	    & \buildrel\triangle\over\longrightarrow  {\rm Hom}_{\Gamma} (F,C^\bullet(X) \otimes C^\bullet(X)) \\	            
										              & \buildrel\cup\over\longrightarrow  {\rm Hom}_{\Gamma} (F,C^\bullet(X))
\end{split}
\end{equation}
where $\cup$ is the co-chain cup product on $C^\bullet(X)$.  The composition respects both the horizontal and vertical filtrations of the double complex ${\rm Hom}_{\Gamma} (F,C^\bullet(X))$.  
This immediately implies Property~\ref{ss1} for the equivariant spectral sequence.  And since $d_r$ is induced from the product on the double complex ${\rm Hom}_{\Gamma} (F,C^\bullet(X))$ which satisfies the Leibniz rule, the differential $d_r$ satisfies the Leibniz rule as well.  We note that in the equivariant spectral sequence, in light of the composition given above, it may be difficult to calculate the products in the derivation. For further details on this construction we direct the reader to Brown \cite{Brown}*{X.4}.     
Having the differential be a derivation is helpful.  In particular, it implies two useful results: 
1) when $d_r(u) = 0$, then $d_r(uv) = u d_r(v)$; and 2) for $r \geq 2$, when $u$ is a class on the $E_r$ page then 
$d_r(u^2) =2ud_r(u) = 0$, as the product on the $E_r$ page is commutative for $r \geq 2$ \cite{Brown}*{X.4.5.iv}. 
\end{rem}

Finally, for both spectral sequences, one can define Steenrod operations which are compatible with differentials as well as the Steenrod squares in the abutment.   A complete account can be found in \cite{Singer}.  
However, we will only need the following facts:

\begin{enumerate}
 \item  There are well-defined operations, $Sq^k: E_r^{p,q} \rightarrow E_r^{p,q+k}$ when $0 \leq k \leq q$ \mbox{\cite{Singer}*{Theorem 2.15}.}
 \item  When $k \leq q - 1$, $d_2 Sq^k u = Sq^k d_2 u$, i.e., the operations commute with the differential  \cite{Singer}*{Theorem 2.17}.
 \item  $Sq^k$ commutes with the appropriate differentials to give the Kudo transgression theorem. \cite{Singer}*{Corollary 2.18}.
\end{enumerate}

\section{The non-central torsion subcomplex} \label{The non-central torsion subcomplex}

In this section we recall the $\ell$--torsion subcomplexes theory of~\cite{AccessingFarrell} and compare it with the
 \textit{non-central} $\ell$--torsion subcomplexes we are going to study. 
We require any discrete group $\Gamma$ 
under our study to be provided with what we will call a \textit{polytopal $\Gamma$--cell complex}, that is, 
a finite-dimensional simplicial complex $X$ with cellular 
$\Gamma$--action such that each cell stabilizer fixes its cell point-wise. 
In practice, we relax the simplicial condition to a polyhedral one,
merging finitely many simplices to a suitable polytope.
We could obtain the simplicial complex back as a triangulation.

\begin{df}
 Let $\ell$ be a prime number. The \emph{$\ell$--torsion subcomplex} of a polytopal $\Gamma$--cell complex~$X$
 consists of all the cells of $X$ whose stabilizers in~$\Gamma$  contain elements of order $\ell$.
\end{df}
We further require that the fixed point set~$X^G$ be acyclic for every nontrivial finite $\ell$--subgroup $G$ of~$\Gamma$.
Then Brown's proposition X.(7.2)~\cite{Brown} specializes as follows. 
\begin{proposition} \label{Brownian}
 There is an isomorphism between the $\ell$--primary parts of the Farrell cohomology of~$\Gamma$ and the
 $\Gamma$--equivariant Farrell cohomology of the $\ell$--torsion subcomplex.
\end{proposition}

For a given Bianchi group, the non-central torsion subcomplex can be quite large.  
It turns out to be useful to reduce this subcomplex, and we identify two conditions under 
which we can do this in a way that Proposition~\ref{Brownian} still holds.
 
\begin{ConditionA} \label{cell condition}
In the $\ell$--torsion subcomplex, let $\sigma$ be a cell of dimension $n-1$ which lies in the boundary of precisely two $n$--cells 
representing different orbits, $\tau_1$ and~$\tau_2$.   Assume further that no higher-dimensional cells of the $\ell$--torsion subcomplex touch $\sigma$;
and that the $n$--cell stabilizers admit an isomorphism
$\Gamma_{\tau_1} \cong \Gamma_{\tau_2}$. 
\end{ConditionA}

\begin{ConditionB} 
The inclusion of the cell stabilizers $\Gamma_{\tau_1}$ and
$\Gamma_{\tau_2}$ into $\Gamma_\sigma$ induces isomorphisms on$\mod \ell$ cohomology.
\end{ConditionB}

When both conditions are satisfied in the $\ell$--torsion subcomplex,  we merge the cells $\tau_1$ and $\tau_2$ along~$\sigma$ and 
do so for their entire orbits.  The effect of this merging is to decrease the size of the  $\ell$--torsion subcomplex without changing its
$\Gamma$--equivariant Farrell cohomology.  This process can often be repeated: by a ``terminal vertex,'' we will denote a vertex with 
no adjacent higher-dimensional cells and precisely one adjacent edge in the quotient space,  and by ``cutting off'' the latter edge,
 we will mean that we remove the edge together with the terminal vertex from our cell complex.
\begin{df}
 A \emph{reduced $\ell$--torsion subcomplex} associated to a polytopal $\Gamma$--cell complex~$X$
 is a cell complex obtained by recursively merging orbit-wise all the pairs of cells satisfying 
 conditions~$\cellCondition$ and~$\weakerCondition$,
 and cutting off edges that admit a terminal vertex when condition~$\weakerCondition$ is satisfied.
\end{df}
The following theorem, stating that Proposition~\ref{Brownian} still holds after reducing, is proved in~\cite{AccessingFarrell}.
\begin{theorem} \label{pivotal}
 There is an isomorphism between the $\ell$--primary parts of the Farrell cohomology of~$\Gamma$ and the
 $\Gamma$--equivariant Farrell cohomology of a reduced $\ell$--torsion subcomplex.
\end{theorem}

In the case of a trivial kernel of the action on the polytopal $\Gamma$--cell complex, 
this allows one to establish general formulae for the Farrell cohomology of~$\Gamma$ \cite{AccessingFarrell}.
In contrast, our action of $\SLO$ on hyperbolic $3$-space has the $2$-torsion group $\{\pm 1\}$ in the kernel; 
since  every cell stabilizer contains  $2$-torsion,
the $2$--torsion subcomplex does not ease our calculation in any way. 
We can remedy this situation by considering the following object, on whose cells we impose a supplementary property.

\begin{df} \label{non-central torsion subcomplex}
 The \emph{non-central $\ell$--torsion subcomplex} of a polytopal $\Gamma$--cell complex $X$
 consists of all the cells of $X$
 whose stabilizers in~$\Gamma$  contain elements of order $\ell$ that are not in the center of~$\Gamma$.
\end{df}

We note that this definition yields a correspondence between, on one side, the \textit{non-central}
$\ell$--torsion subcomplex for a group action with kernel the center of the group,
 and on the other side, the $\ell$--torsion subcomplex for its central quotient group.
 We use this correspondence in order to identify the \textit{non-central} $\ell$--torsion subcomplex for the action of $\SLO$
 on hyperbolic $3$--space as the $\ell$--torsion subcomplex of $\PSLO$.  
  However, incorporating the non-central condition for $\SLO$ introduces significant technical obstacles, 
  which we address in Section~\ref{Section:E2 page}.

We recall the following information from~\cites{AccessingFarrell} on the $\ell$--torsion subcomplex of $\PSLO$. Let~$\Gamma$ be a finite index subgroup in $\text{PSL}_2(\mathcal{O}_{-m})$. 
Then any element of~$\Gamma$ fixing a point inside hyperbolic $3$-space~$\hyper$ acts as a rotation of finite order.
By Felix Klein's work, we know conversely that any torsion element~$\alpha$ is elliptic and hence fixes some geodesic line.
We call this line \emph{the rotation axis of~$\alpha$}.
Every torsion element acts as the stabilizer of a line conjugate to one passing through the Bianchi fundamental polyhedron.
We obtain the \textit{refined cellular complex} from the action of~$\Gamma$ on~$\hyper$
 as described in~\cite{RahmTorsion}, 
namely we subdivide~$\hyper$  until the stabilizer in~$\Gamma$ of any cell $\sigma$ fixes $\sigma$ point-wise. 
We achieve this by computing Bianchi's fundamental polyhedron for the action of~$\Gamma$,
 taking as a preliminary set of 2-cells its facets lying on the Euclidean hemispheres
 and vertical planes of the upper-half space model for $\hyper$,
 and then subdividing along the rotation axes of the elements of~$\Gamma$. 
 
It is well-known~\cite{SchwermerVogtmann} that if $\gamma$ is an element of finite order $n$ in a Bianchi group, then $n$ must be 1, 2, 3, 4 or 6,
 because $\gamma$ has eigenvalues $\rho$ and $\overline{\rho}$,
 with $\rho$ a primitive $n$--th root of unity, and the trace of~$\gamma$ is $\rho + \overline{\rho} \in \calO_{-m} \cap {\mathbb{R}} = \Z$.   
 When $\ell$ is one of the two occurring prime numbers $2$ and~$3$, the orbit space of this sub-complex is a graph,
 because the cells of dimension greater \mbox{than 1} are trivially stabilized in the refined cellular complex.
We can see that this graph is finite either from the finiteness of the Bianchi fundamental polyhedron, 
or from studying conjugacy classes of finite subgroups as in~\cite{Kraemer}.
 
As in \cite{RahmFuchs}, we make use of a $2$-dimensional deformation retract $X$ of the refined cellular complex, 
equivariant with respect to a Bianchi group \mbox{$\Gamma := \SLO$}.  This retract has a cell structure 
in which each cell stabilizer fixes its cell point-wise.  
Since $X$ is a deformation retract of $\hyper$ and hence acyclic,  $\Cohomol^*_\Gamma(X) \cong \Cohomol^*_\Gamma(\hyper) \cong \Cohomol^*(\Gamma)$.
In what follows, we need to know about the subgroups of finite order in Bianchi groups, as these appear as cell stabilizers.  The subgroups are given in Table \ref{table:covering} for both the $\SLtwo$ and $\PSLtwo$ Bianchi groups. 
 
\begin{caveat}
 In order for the non-central $2$-torsion subcomplex to be contained in the deformation retract $X$,
 we exclude $\calO_{-m}$ from being the ring of Gaussian integers $\calO_{-1}$ for the entirety of this article.
 This does not cause any harm, as in this special case, all of the cohomology can be computed by hand~\cite{SchwermerVogtmann}.  
\end{caveat}

From two theorems of Norbert Kr\"amer, 
we deduce that there are only four types of connected components of non-central reduced $2$--torsion subcomplex quotients possible for the  $\PSLtwo$,  
and hence $\SLtwo$, Bianchi groups.  
This is the subject of the following corollary.  
 
\begin{table}
\begin{center}
$
\begin{array}{|cc|c|c|c|c|}
\hline \SLO & \text{contains}           &\Z/2n & \Q    &\Di    &\Te\\
\hline \PSLO &  \text{contains}   & \Z/n & \Dtwo &\Dthree& \Af \\
\hline 
\end{array}
$
\end{center} 
  
\caption{  The finite subgroups of $\SLO$ and their quotients in $\PSLO$ by the order-2-group $\{1, -1\}$.   Here, $\Z/n$ is the cyclic group of order $n$,  $\Q$   the quaternion group of order $8$,  $\Di$ the dicyclic group of order $12$, $\Te$ the binary tetrahedral group of order $24$.   Further, the dihedral group are $\Dtwo$ with four elements and $\Dthree$ with six elements,  and the tetrahedral group is isomorphic to the alternating group $\Af$ on four letters.} \label{table:covering}
\end{table}

\begin{corollary}[to theorems of Kr\"amer] \label{obs:4cases}
The shapes of the types of connected components of non-central reduced $2$--torsion subcomplex quotients, 
with stabilizers for the $\SLtwo$ case, 
are all shown in the last column of Table~\ref{table:subcomplexes}.
\end{corollary}

\begin{proof}
For this proof, we consider Bianchi groups as being $\PSLO$, unless explicitly stated otherwise.
Kr\"amer [\cite{Kraemer}, Satz 8.3 and Satz 8.4] has shown that the finite subgroups of $\PSLO$, particularly the $2$-dihedral subgroups $\Dtwo$,
are related to each other in one of only a few ways.  For each $2$--dihedral subgroup $G$, there is, up to conjugacy,
precisely one $2$--dihedral subgroup $H,$ such that the intersection of $G$ and $H$ is a group $C$ of order 2
and such that $H$ is not conjugate to $G$.   Let $C'$, respectively $C''$, be the two other subgroups of $G$ of order 2. 
There are three possible cases.  
\begin{itemize}
 \item[($\theta$).] The groups $G$ and $H$ are maximal finite subgroups of the Bianchi group.  In particular, neither $G$ nor $H$ are contained in an $\Af$ subgroup.
Then, the groups $C$, $C'$ and $C''$ are pairwise non-conjugate to each other;
up to conjugacy, there is precisely one $2$--dihedral subgroup $H'$ which contains $C'$
and is not conjugate to $G$;
and up to conjugacy, there is precisely one $2$--dihedral subgroup $H''$ which contains $C''$
and is not conjugate to $G$. Furthermore, the groups $H$, $H'$, and $H''$ are conjugate to each other.
\item[($\rho$).] The group $G$ is a maximal finite subgroup of the Bianchi group, but $H$ is contained in a copy of $\Af$. 
Then the groups $C'$ and $C''$ are conjugate to each other, but not conjugate to $C$.  In addition, 
up to conjugacy, $G$ is the only $2$--dihedral subgroup which contains $C'$ and the only $2$--dihedral subgroup which contains $C''$.
\item[($\iota$).] The groups $G$ and $H$ are both contained in copies of $\Af$ in the Bianchi group.  Then the groups $C$, $C'$ and $C''$ are conjugate to each other.
\end{itemize}

Using the description of a reduced $2$-torsion subcomplex for the action of $\PSLO$ 
on hyperbolic space given in~\cite{AccessingFarrell} 
in terms of conjugacy classes of finite subgroups of $\PSLO$,
the three cases yield the following types of connected components in the quotient 
of a reduced $2$-torsion subcomplex.
\begin{itemize}
 \item[($\theta$).] The conjugacy class of $G$ corresponds to one bifurcation point in the two torsion subcomplex,
 and the conjugacy class of $H$, $H'$, and $H''$ yields another.  These two bifurcation points are connected by the three edges coming from $C$, $C'$ and $C''$.
This closes the connected component and gives it the shape $\graphFive$.
\item[($\rho$).] The group $H$ corresponds to an endpoint connected by one edge (coming from $C$) to the bifurcation point corresponding to ~$G$.
The conjugacy class containing $C'$ and $C''$ yields a second edge connecting the latter bifurcation point to itself.
This closes the connected component and gives it the shape $\graphTwo$.
\item[($\iota$).] The conjugacy class containing $C$, $C'$ and $C''$ corresponds to a single edge, 
connecting its two end-points coming from $G$ and $H$.
This closes the connected component and gives it the shape $\edgegraph$.
\end{itemize}
As was observed in \cite{RahmTorsion},  conjugacy classes of cyclic subgroups which are not contained in any dihedral subgroup of $\PSLO$ yield a connected component of shape $\circlegraph$.

Passing to the preimages of the projection from $\SLO$ to $\PSLO$ as indicated in Table~\ref{table:covering}, 
we obtain the last column of Table~\ref{table:subcomplexes}, with stabilizers as claimed.
\end{proof}

All four shapes subject to the above corollary occur at small discriminant absolute values 
(see Appendix \ref{Numerical results} and a few more examples in Figure 3 of~\cite{RahmTorsion}). 
Joint work in progress of Grant Lakeland with the authors indicates the existence of one more component type for congruence subgroups in the Bianchi groups.

A connected component of the $2$--torsion subcomplex can be considered as a tree with 
action of the subgroup $G$ of the Bianchi group that sends the tree to itself.
We obtain $G$ as the \emph{groupe fondamental du graphe de groupes} (\cite{Serre-arbres}, 
not to be confused with the fundamental group of the underlying graph) 
of a connected component of the quotient of a reduced non-central $2$--torsion subcomplex
with attached stabilizer groups and monomorphisms from the edge stabilizers into the vertex stabilizers.
Such $G$ in this paper can be built iteratively using amalgamated products and HNN extensions.
We collect the four possible connected component types in Table~\ref{table:subcomplexes}.   
We note that the degree of each vertex in the quotient space is the same as the number of distinct conjugacy classes of $\Z/4$ in the vertex stabilizer.  
This information determines the HNN extensions up to ordering of the conjugacy classes.  
We also observe that all stabilizers which contain a copy of $\Q$ are associated to vertices.  
This observation will be used in Section \ref{Section:E2 page}.

\begin{table}
 \begin{center} 
 \begin{tabular}{|c|c|c|}
 \hline & & \\
${}$  & Type of group $G$&  Quotient of tree acted on by $G$ 
\\  \hline & & \\
($o$) & HNN extension $\Z/4 \ast_{\Z/4}$ & $\circlegraph \thinspace \Z/4$ \\
 & & \\
($\iota$) & Amalgamated product $\Te \ast_{\Z/4} \Te$ & $\Te \edgegraph \Te$ \\
 & & \\
($\theta$) & Double HNN extension $\left( \left( \Q \ast_{\Z/4} \Q \right) \ast_{\Z/4} \right) \ast_{\Z/4}$ & $\Q \graphFive \thinspace \Q$ \\
 & & \\
($\rho$) & Iterated construction $(\Q \ast_{\Z/4}) \ast_{\Z/4} \Te$ & $\Q \graphTwo \Te$ \\
\hline
\end{tabular}
\end{center}

\medskip

\caption{Connected component types of non-central reduced $2$--torsion subcomplex quotients.
The homotopy type of the quotient space is given in the right-most column.  
All  edges have stabilizer type $\Z/4$, and the stabilizer types of the vertices are given.} \label{table:subcomplexes}
\end{table}

\section{Maps induced by finite subgroups in the Bianchi groups}\label{sec:pieces}

In this section, we classify the possible cell stabilizers of the 
$\SLO$-action and determine the restriction maps between subgroups.  
We will use this information to determine the cohomology of components of 
reduced non-central $2$-torsion subcomplexes in Section \ref{Calculation of the non-polynomial part of the cohomology}.
 
Since the action of $\SLO$  on $X$ is properly discontinuous, cell stabilizers are finite subgroups of $\SLO$.
The enumeration of the finite subgroups of $\SLtwo(\mathbb{C})$ is a classical result,
 and the list of finite subgroups which appear in the Bianchi groups $\PSLtwo(\ringOm)$
 is also well known~\cite{binaereFormenMathAnn9}, see Table~\ref{table:covering}, 
 in which we fix our notations.  

We start by recalling some mod $2$ cohomology rings.  In what follows, the symbol $\ef$ represents the field of two elements,
and all the cohomology rings and groups in which we omit the coefficients 
are meant to be taken with (obviously trivial) $\ef$--coefficients.
We write reduced cohomology classes [in square brackets] and nilpotent cohomology classes (in parentheses).
The index of a class specifies its degree.

\begin{proposition} \cite{AdemMilgram} \label{thm:startcoh}
The cohomology rings for finite subgroups of $\PSLO$ are given below, where a subscript denotes the degree of the generator.
\begin{alignat}{3}
  &  \Cohomol^*(\Z/2)  & \cong &  \Cohomol^*(\Dthree)  \cong \ef[x_1]  \nonumber \\
  &  \Cohomol^*(\Z/3)  & \cong &  \Cohomol^*(1)  \cong  \ef  \nonumber\\
  &  \Cohomol^*(\Dtwo)  & \cong  &   \ef[x_1,y_1] \nonumber \\
  &  \Cohomol^*(\Af)  & \cong  &  \ef[u_2,v_3,w_3]\/ 
                  / \langle u_2^3 + v_3^2 + w_3^2 + v_3w_3 \rangle  \nonumber  
\end{alignat}
\qed
\end{proposition}

We use the  Lyndon--Hochschild--Serre spectral sequence to determine the cohomology rings of the finite subgroups of $\SLtwo(\ringOm)$. 
 In particular, for each finite subgroup $G$ of $\SLtwo(\ringOm)$,
 there is a corresponding finite subgroup $Q$ of $\PSLtwo(\ringOm)$ which fits into the central extension 
\[
1 \rightarrow \Z/2 \longrightarrow G \longrightarrow Q \rightarrow 1
\]
where $\Z/2$ is the subgroup of $\SLtwo(\ringOm)$ containing $\{ \pm I \}$.
  We note that all finite subgroups of $\SLtwo(\mathbb{C})$ also sit inside ${\rm SU}(2)$ and act freely on it.
  Since ${\rm SU}(2)$ can be identified with the $3$--sphere,
 the  cohomology rings for the finite subgroups of $\SLtwo(\ringOm)$ are periodic of period dividing $4$ \cite{Brown}.
  In particular, the cohomology rings can all be expressed as a tensor product where one term is a  
  polynomial ring on one generator.  

\begin{proposition}  \label{prop:SLfincoh}      
The following are the mod $2$ cohomology rings of the finite subgroups of $\SLtwo(\ringOm)$.
\begin{alignat}{3}
  &  \Cohomol^*(\Z/4)  & \cong &   \Cohomol^*(\Di)     \cong \ef[e_2](b_1)  \nonumber \\
  &  \Cohomol^*(\Z/2)  & \cong &   \Cohomol^*(\Z/6)  \cong  \ef[e_1]  \nonumber \\
  &  \Cohomol^*(\Q)    & \cong  &  \ef[e_4](x_1, y_1, x_2, y_2, x_3) \nonumber \\
  &  \Cohomol^*(\Te)   & \cong  &  \ef[e_4](b_3) \nonumber
\end{alignat}
\end{proposition}

\begin{proof}
The cohomology results for cyclic group results are straightforward, 
and the calculation of the other cohomology rings (and a classification for all periodic groups)
are contained in~\cite{AdemMilgram}.  
However, we briefly review the derivations here as the descriptions will be useful when we determine restriction maps to subgroups.

The dicyclic group $\Di$ is a semidirect product and fits into the short exact sequence
\[
1 \rightarrow \Z/3 \longrightarrow \Z/3 \rtimes \Z/4 \longrightarrow \Z/4 \rightarrow 1.
\]  
Since $\Cohomol^*(\Z/3)$ has trivial mod $2$ homology, in the  Lyndon--Hochschild--Serre spectral sequence the only nontrivial  cohomology 
 occurs on the horizontal axis and is isomorphic to $\Cohomol^*(\Z/4)$.

For $\Z/4$, $\Q$ and $\Te$, we again use the  Lyndon--Hochschild--Serre spectral sequence for central extensions.  In the spectral sequence of the extension,
 let $\ef[e_1]$ be the cohomology ring corresponding to the central subgroup $\Z/2$.  On the $E_2$ page of the spectral sequence, the image of $d_2(e_1)$ can be identified with the $k$--invariant.  In addition, the Kudo transgression theorem describes the images of $d_{2^n+1}\left((e_1)^{2^n}\right)$.  

We start with the calculation of $\Cohomol^*(\Z/4)$.  
Let $\Cohomol^*(\Z/2) \cong \ef[x_1]$, and consider the spectral sequence associated to the central extension. 
 In this case, the $k$--invariant $d_2(e_1) = x_1^2$, so $d_2(e_1 x_i^k) = x_i^{k+2}$ and $d_2(e_1^2) = 0$.
  As the result, the spectral sequence collapses at the $E_2$ page with classes in even rows and in columns $0$ and $1$ only. 
 The class $e_1^2$ represents the polynomial class in $\Cohomol^2(\Z/4)$,
 and the class in $E_2^{0,1}$ is the one-dimensional exterior class.

We next calculate $\Cohomol^*(\Q)$.  The extension is central, so $E_1^{p,q} \cong \Cohomol^p(\Dtwo) \otimes \Cohomol^q(\Z/2)$, and by~\cite{AdemMilgram}*{Proposition IV.2.10},
 the $k$--invariant $d_2(e_1) = x_1^2 +x_1y_1 +y_1^2 \in \Cohomol^*(\Dtwo)$.  Also, $d_2(e_1^2) = 0$, which completely determines the differentials on the $E_2$ page.  Consequently, the only classes which survive to the $E_3$ page lie in even rows.  Next, using the Kudo transgression theorem, 
\[
d_3(e_1^2)  = d_3\left(Sq^1(e_1)\right)  = Sq^1(x_1^2 + x_1 y_1 + y_1^2) = x_1^2y_1 + x_1 y_1^2.  
\]
Through careful accounting, one can show that the classes that remain in $E_3$ are in columns $0$ through~$3$,  and that the spectral sequence has non-zero terms in rows congruent to $0 \text{ mod } 4$.    Consequently, all higher differentials vanish and  $E_3 \cong E_{\infty}$.    The only classes left correspond to $1$, $x_1$, $y_1$, $x_1^2$, $y_1^2$, $x_1^2 y_1 = x_1 y_1^2$  and the product of these classes with powers of $e_1^4$, which represents a four-dimensional polynomial class.  
We find that as a ring, $\Cohomol^*(\Q) \cong  \ef[e_4](x_1, y_1) / \langle R \rangle$, where $R$ is the ideal generated by $x_1^2 +  x_1y_1 + y_1^2$ and $x_1^2 y_1 + x_1 y_1^2$.   The result follows.

Finally, we consider the calculation of $\Cohomol^*(\Te)$, which proceeds in a similar manner as the case for $\Cohomol^*(\Q)$,
 except now the quotient group is $\Af$ and we have to consider
 \[
\Cohomol^*(\Af) \cong \ef[u_2,v_3,w_3]\ / \langle u_2^3 + v_3^2 + w_3^2 + v_3w_3 \rangle.
\]
Briefly, $d_2(e_1) = u_2$, and $d_3(e_1^2) = v_3 = Sq^1(u_2)$ \cite{BerkoveMod2}.  Once again, after a careful check, the spectral sequence collapses at the $E_3$ page and the only classes left correspond to $1$, the four-dimensional class represented by $e_1^4$, $w_3 \in E_3^{3,0}$, and the product of $w_3$ with powers of $e_1^4$.
\end{proof}

Given the following commuting diagram of groups, 
\[
\begin{CD}
\Z/2 @>>>  G_1 @>>>    Q_1     \\
        @|               @VVi_2V      @VVi_1V   \\
\Z/2 @>>>  G_2 @>>>   Q_2      \\
\end{CD}
\]
if one knows the induced map $i_1^* : \Cohomol^*(Q_2) \rightarrow \Cohomol^*(Q_1)$, it is often possible to calculate the effect of $i_2^*: \Cohomol^*(G_2) \rightarrow \Cohomol^*(G_1)$.  
  The commutative diagram gives rise to two Lyndon--Hochschild--Serre spectral sequences, one for each extension, and the two are compatible via the isomorphism on the fiber.  Specifically, the maps between the spectral sequences are given by $i_1^*$ along the $x$--axis and the identity along the $y$--axis.  Once the spectral sequences converge, for the cases we consider in this work the result gives the effect of the restriction map $res^{G_2}_{G_1} = i_2^*$ for classes on the edges.

\begin{proposition}\label {prop:restriction}
The following are the nontrivial restriction maps on cohomology generators for finite subgroups of $\SLtwo(\ringOm)$:
\[
\begin{array}{ll}
res^{\Z/4}_{\Z/2} (e_2) & = e_1^2  \\
res^{\Di}_{\Z/6}(e_2) & = res^{\Di}_{\Z/2}(e_2) = e_1^2 \\
res^{\Q}_{\Z/4}(e_4) & = e_2^2  , \quad res^{\Q}_{\Z/4}(x_1)  = b_1  \\   
res^{\Q}_{\Z/2}(e_4) & = e_1^4  \\
res^{\Te}_{\Z/6}(e_4) & = res^{\Te}_{\Z/2}(e_4) = e_1^4  \\
res^{\Te}_{\Z/4}(e_4) & = e_2^2
\end{array}{}
\]
In addition, $ res^{\Z/6}_{\Z/2}$ and $res^{\Di}_{\Z/4}$ are isomorphisms.
\end{proposition}

The results in this proposition are well-known, for example, see Lemma 2.11 and Corollary~6.7 in~\cite{AdemMilgram}.  
We also note that there are three copies of $\Z/4$ in $\Q$.  
The three corresponding injections are in bijective correspondence with the three surjections from 
$(\Z/2)^2 \cong H^1(\Q)$ to $\Z/2 \cong H^1(\Z/4)$.


\begin{rem} \label{rem:separatepolyext}
We note that the restriction map is natural with respect to cup products, so the results of Proposition~\ref {prop:restriction} can be used to determine the effect of the restriction map on other cohomology classes. 
 In particular, the only nontrivial restriction map on classes with nilpotent components is
 $ res^{\Q}_{\Z/4}(e_4^i x_1)  = e_2^{2i} b_1$.
 In addition, in the cohomology of the groups we consider,
 only polynomial classes restrict non-trivially to polynomial classes.
\end{rem}


We noted in the discussion after Proposition~\ref{thm:startcoh} that the cohomology of the finite subgroups of $\SLtwo(\ringOm)$
 are all periodic of period dividing $4$. 
 Above the virtual cohomological dimension, which is 2 for the Bianchi groups, the same period can be observed for the
 $\SLtwo$ Bianchi groups and their subgroups.  The proof of this is based on ideas of \cite{AdemSmith} and \cite{CP}.
 
\begin{proposition}\label{prop:Bianchiperiodicity}
Any (not necessarily finite) subgroup $\Gamma$ in $\SLtwo(\ringOm)$ has periodic cohomology above the virtual cohomological dimension. 
 The periodicity can be realized by cup product with a $4$--dimensional class $\alpha$.  
\end{proposition}

\begin{proof}
The Bianchi groups act properly on $X$, which is homeomorphic to $\mathbb{R}^3$.  In addition, from the discussion after Proposition~\ref{thm:startcoh}, 
 we know that the finite subgroups in the Bianchi groups act freely on  ${\rm SU}(2) \subset \SLtwo(\mathbb{C})$.   Therefore the groups $\SLtwo(\ringOm)$ and their subgroups act freely and properly on $\mathbb{R}^3 \times \sphere^3$. 
 
 We essentially follow the argument in \cite{CP} and set $Y := \mathbb{R}^3 \times \sphere^3$  with the $\Gamma$--action described above.  There is a fibration
\[
  Y \rightarrow E\Gamma \times_{\Gamma} Y \rightarrow B\Gamma,
\]
where $B\Gamma$ is the classifying space for~$\Gamma$ and $E\Gamma$ is its universal cover.  
Since $Y$ has the homotopy type of $\sphere^3$, this is a spherical fibration which
is orientable as we are working with $\F_2$--coefficients.  The result follows by applying the Gysin exact sequence.

The class $\alpha$ is the image under $d_4$ of the generator of $\Cohomol^3 (\sphere^3)$.  
As mentioned in the introduction, the referee has pointed out that this class is the second Chern class of the natural representation of $\mathrm{SL}_2(\mathbb{C})$.    
\end{proof}

\section{Calculation of the $E_2$ page} \label{Section:E2 page}

In this section, we describe the $E_2$ page of the equivariant spectral sequence for Bianchi groups in a fairly general way.  Some of our analyses will hinge on distinguishing types of cohomology classes.  Given a finite group $G$, we note that the nilpotent classes in $\Cohomol^*(G)$ form an ideal.
We will  denote this ideal by $\Home{*}(G)$.  
We next define the reduced  quotient module $\Homp{*}(G)$ 
as the quotient by the ideal of nilpotent classes, so
\[
 0 \to  \Home{*}(G) \to \Cohomol^*(G) \to \Homp{*}(G) \to 0
 \]
is exact.  In what follows, we will also use the term ``reduced class'' to mean a class in the cohomology ring which has non-trivial image in the quotient module.

Denote by $X$ the cell complex described in Section~\ref{The non-central torsion subcomplex} before Corollary~\ref{obs:4cases}, and 
by $X_s$ the non-central $2$--torsion subcomplex of $X$ with respect to $\Gamma$.
Further, denote by $\xsp$ the subcomplex of $X_s$ consisting of cells whose stabilizer group contains a copy of $\Q$,
i.e. the stabilizer being of type either $\Q$ or $\Te$. 
We note that if a cell of $X$ is not in $\xsp$, then by Proposition~\ref{prop:SLfincoh}     
the cohomology of that cell's stabilizer is isomorphic to $\Cohomol^*(\Z/4)$ 
if the cell is in $X_s$,
respectively to $\Cohomol^*(\Z/2)$ if the cell is not in $X_s$.  
We also note that $\xsp$ is a $0$-dimensional subcomplex of $X$. 

\begin{note} \label{hypo} 
In the calculations that follow, we require that for any cell $\tau \in X_s$, the inclusion of the 
center  $\Center \hookrightarrow \Gamma_{\tau}$
 into the cell stabilizer of $\tau$ induces a monomorphism
on the reduced parts of the cohomology rings, \mbox{$\Homp{*}(\Gamma_{\tau}) \hookrightarrow \Homp{*}(\Center)$.}
Furthermore, for any cell $\sigma \in X$ not in $X_s$, we want that the inclusion $\Center \hookrightarrow \Gamma_{\sigma}$
induces an isomorphism \mbox{$\Homp{*}(\Gamma_{\sigma}) \to \Homp{*}(\Center)$.}
By Proposition~\ref {prop:restriction}, the action of the Bianchi groups on the cell complex $X$ obtained from hyperbolic space satisfies these conditions.
\end{note}

In the following equivariant spectral sequence material, we will assume that we are working with $\Gamma$-equivariant cohomology, 
i.e., $E_2^{p,q}(Y)$  stands for the $E_2^{p,q}$-term of the equivariant spectral sequence associated to the action of $\Gamma$ on $Y$,
unless specified otherwise.
The next theorem will be stated in terms from a relative spectral sequence, 
which is defined as follows.  
Given a cellular subcomplex $X^\prime \subseteq X$,
there is a short exact sequence of co-chain complexes
 \[
 0 \rightarrow C^*(X, X^\prime) \rightarrow C^*(X) \rightarrow C^*(X^\prime) \rightarrow 0.
 \]
Let $F$ be a free resolution for $\Gamma$. 
Applying ${\rm Hom}_{\ef[\Gamma]} (F,-)$ 
and then cohomology yields a short exact sequence of chain complexes,
\begin{equation} \label{relativeSpectralSequence}
0 \rightarrow E_1^{*,q}(X, X^\prime) \rightarrow  E_1^{*,q}(X) \rightarrow  E_1^{*,q}(X^\prime) \rightarrow 0.
\end{equation}

\begin{thm}\label{thm:E2page}
In the equivariant spectral sequence with $\F_2$-coefficients converging to 
$\Cohomol^{p+q}_\Gamma(X)$, \\
the $E_2^{p,q}$ page is given by the following rows, 
$k$ running through ${\mathbb{N}} \cup \{0\}$:
\[
\begin{array}{l  l}
  E_2^{p,4k+3}(X) \cong E_2^{p,3}(X_s)        \oplus  E_2^{p,3}(X, X_s),   \\
 E_2^{0,4k+2}(X) \cong \Cohomol^2_\Gamma(\xsp)  \oplus  E_2^{0,2}(X, X_s^\prime), &  E_2^{p,4k+2}(X) \cong E_2^{p,2}(X, X_s^\prime) \text{ for } p \geq 1,\\
 E_2^{p,4k+1}(X) \cong E_2^{p,1}(X_s)        \oplus  E_2^{p,1}(X, X_s),  \\
E_2^{p,4k}(X) \cong \Cohomol^p(_\Gamma \backslash X).   \\
\end{array}
\]
\end{thm} 
 By the periodicity established in Proposition~\ref{prop:Bianchiperiodicity}, to prove this theorem 
 it is sufficient to calculate only the first four rows. 

\begin{proof}[Proof in odd degrees $q$]
 For the inclusion $X_s \subset X$, where $\dim X_s = 1$ and $\dim X = 2$,
 Sequence~(\ref{relativeSpectralSequence}) is concentrated in the following diagram.
\begin{equation} \label{cellularSES}
 \begin{tikzpicture}[descr/.style={fill=white,inner sep=1.5pt}]
        \matrix (m) [
            matrix of math nodes,
            row sep=1.4em,
            column sep=1.4em,
            text height=3.95ex, text depth=2.95ex
        ]
        {{} & 0   & 0 & 0  & {} \\
        0 &  E_1^{2,q}(X, X_s) &  E_1^{2,q}(X) & 0 & 0  \\        
        0 &  E_1^{1,q}(X, X_s) &  E_1^{1,q}(X) & E_1^{1,q}(X_s) & 0  \\    
        0 &  E_1^{0,q}(X, X_s) &  E_1^{0,q}(X) & E_1^{0,q}(X_s) & 0  \\    
        {} & 0 & 0 & 0.  & {} \\
        };
        \path[overlay,->, font=\scriptsize,>=latex]
        (m-2-2) edge node[descr,xshift=-4.4ex] {$d_1^{2,q}$}  (m-1-2)
        (m-2-3) edge node[descr,xshift=-4.4ex] {$d_1^{2,q}$}  (m-1-3)
        (m-2-4) edge node[descr,xshift=-4.4ex] {$d_1^{2,q}$} (m-1-4)
        (m-2-1) edge (m-2-2)
	(m-2-2) edge (m-2-3) 
        (m-2-3) edge (m-2-4)
        (m-2-4) edge (m-2-5)
	(m-3-2) edge node[descr,xshift=-4.4ex] {$d_1^{1,q}$} (m-2-2)
        (m-3-3) edge node[descr,xshift=-4.4ex]   {$d_1^{1,q}$}  (m-2-3)
        (m-3-4) edge node[descr,xshift=-4.4ex]   {$d_1^{1,q}$}  (m-2-4) 
        (m-3-1) edge (m-3-2)
	(m-3-2) edge (m-3-3) 
        (m-3-3) edge (m-3-4)
        (m-3-4) edge (m-3-5)
	(m-4-2) edge node[descr,xshift=-4.4ex]  {$d_1^{0,q}$} (m-3-2)
        (m-4-3) edge node[descr,xshift=-4.4ex]    {$d_1^{0,q}$} (m-3-3)
        (m-4-4) edge node[descr,xshift=-4.4ex]    {$d_1^{0,q}$} (m-3-4) 
        (m-4-1) edge (m-4-2)
	(m-4-2) edge (m-4-3) 
        (m-4-3) edge (m-4-4)
        (m-4-4) edge (m-4-5)
	(m-5-2) edge (m-4-2)
        (m-5-3) edge  (m-4-3)
        (m-5-4) edge (m-4-4) 
  ;
\end{tikzpicture} 
\end{equation}\normalsize
Then, for fixed $p$, we have a splitting of $\ef$-modules,
$E_1^{p,q}(X) \cong E_1^{p,q}(X_s) \oplus E_1^{p,q}(X, X_s)$.
Now with respect to this splitting, there are no nontrivial maps between $E_1^{p,q}(X_s)$ 
and $E_1^{p+1,q}(X, X_s)$ because of the following mismatch:  
when $q$ is odd, classes in $E_2^{p,q}(X_s)$ are nilpotent by Proposition~\ref{prop:SLfincoh}.  
On the other hand, classes in the relative cohomology group $E_2^{p+1,q}(X, X_s)$ 
are all reduced, since any cell stabilizer of a cell not in $X_s$ 
has mod-$2$ cohomology isomorphic to $\Cohomol^*(\Z/2)$. 
We note that the coboundary operator commutes with Steenrod operations.  Then  Proposition~\ref{prop:SLfincoh} 
and the application of either $Sq^2 Sq^1$ or $Sq^3$ implies the vanishing result for maps between $E_1^{p,q}(X_s)$ 
and $E_1^{p+1,q}(X, X_s)$.  Finally, we remark that the map from $E_1^{p,q}(X)$ to $E_1^{p,q}(X_s)$ takes 
nilpotent classes to nilpotent classes. 

Therefore, the $d_1^{p,q}$ differentials split over $E_1^{p,q}(X) \cong E_1^{p,q}(X_s) \oplus E_1^{p,q}(X, X_s)$.
So Sequence~(\ref{cellularSES}) splits not only level-wise,
but as a short exact sequence of chain complexes.
Taking homology with respect to $d_1^{p,q}$ then yields the desired splitting,
\[
 E_2^{p,q}(X) \cong E_2^{p,q}(X_s) \oplus E_2^{p,q}(X, X_s).
 \]
\end{proof}

\begin{proof}[Proof in degrees $q = 4k+2$]
For the inclusion $\xsp \subset X$,  Sequence~(\ref{relativeSpectralSequence})
becomes the short exact sequence of chain complexes
\begin{equation} \label{subsubSES}
0 \to E_1^{p,2}(X, \xsp) \to E_1^{p,2}(X) \to E_1^{p,2}(\xsp) \to 0.
\end{equation}
Recall that $\xsp$ is the subcomplex of $X_s$ 
consisting of cells whose stabilizer group contains a copy of $\Q$,
and such cells are $0$-dimensional. 
So $E_1^{p,2}(\xsp)$ is concentrated in the module $E_1^{0,2}$.  
In addition, classes in $E_1^{0,2}(\xsp)$ are nilpotent whereas
classes in $E_1^{1,2}(X,\xsp)$, being associated to edges, are reduced
--- all of the edge stabilizers are isomorphic to $\Z/2$ or $\Z/4$ 
(the latter occurring when the edge is part of the non-central $2$-torsion subcomplex).  
The splitting argument is now the same as in the case for $q$ odd, except for that we use $Sq^2$.
So Sequence~(\ref{subsubSES}) splits, not only level-wise, 
but as a short exact sequence of chain complexes.
Finally, as $\xsp$ is $0$-dimensional, $E_1^{p,2}(\xsp) \cong E_{\infty}^{p,2}(\xsp)$
is concentrated in column $p = 0$, where it is isomorphic to $\Cohomol^2_\Gamma(\xsp)$.
\end{proof}  
 
\begin{proof}[Proof in degrees $q = 4k$]

The $E_1$ page of the spectral sequence is the co-chain complex 
\[ 
0 \to \bigoplus\limits_{\sigma_0 \in  _\Gamma\backslash X^{(0)}} \Cohomol^{4k}(\Gamma_{\sigma_0}) \to 
         \bigoplus\limits_{\sigma_1 \in  _\Gamma\backslash X^{(1)}} \Cohomol^{4k}(\Gamma_{\sigma_1}) \to
         \bigoplus\limits_{\sigma_2 \in  _\Gamma\backslash X^{(2)}} \Cohomol^{4k}(\Gamma_{\sigma_2}) \to  0 
 \]
Furthermore,  by Proposition~\ref{prop:SLfincoh}, $\Cohomol^{4k}(\Gamma_{\sigma_i}) \cong \ef$ and the horizontal maps are induced by cell inclusion.  Taking the homology of the co-chain complexes yields the isomorphism
$$ E_2^{p,4k}(X) \cong \Cohomol^{p}(_\Gamma \backslash X).$$
 \end{proof}

 Theorem~\ref{thm:E2page} describes the rows of the $E_2$ page in terms of rows of $E_2$ pages of subcomplexes and relative complexes.  The cohomology of the former terms will be
 calculated in Section \ref{Calculation of the non-polynomial part of the cohomology}.  We calculate the latter terms next, first establishing some notation.

 \begin{notation}\label{not:Euler}
 Here and in what follows,  $\beta_q(_\Gamma \backslash X_s) := \dim_{\rationals}\Cohomol_q(_\Gamma \backslash X_s ; \thinspace \rationals)$.  Denote by $\chi(_\Gamma \backslash X_s) = \beta_0(_\Gamma \backslash X_s) -\beta_1(_\Gamma \backslash X_s)$ 
 the Euler characteristic of the orbit space $_\Gamma \backslash X_s$ of the non-central $2$--torsion subcomplex.
\end{notation}

For $\calO$ not the Gaussian or Eisensteinian integers, the Euler characteristic of the hyperbolic orbit space $_\Gamma \backslash X$ vanishes
 because its boundary consists of disjoint $2$--tori \cite{Serre}*{p. 513}.
 Consequently, the Betti numbers $\beta_1$ and $\beta_2$ of the hyperbolic orbit space $_\Gamma \backslash X$ satisfy $1 -\beta_1 +\beta_2 = 0$,
 so we can replace $\beta_{2}$  by $\beta_{1} -1$ when it is convenient.
\begin{notation}
 For $q \in \{1, 2\}$, denote the dimension $\dim_{\F_2}\Cohomol^q(_\Gamma \backslash X ; \thinspace \F_2)$ by $\beta^q$.
\end{notation}
 The Universal Coefficient Theorem yields 
 $$ \Cohomol^q(_\Gamma \backslash X ; \thinspace \F_2) \cong {\rm Hom}(\Cohomol_q(_\Gamma \backslash X ; \thinspace \Z), \thinspace \F_2) 
 \oplus {\rm Ext}(\Cohomol_{q-1}(_\Gamma \backslash X ; \thinspace \Z), \thinspace \F_2).$$
 As  $X$ is $2$-dimensional, the group $\Cohomol_2(_\Gamma \backslash X ; \thinspace \Z)$ contains no torsion, 
 so we obtain that the dimension $\beta^2 $ 
 equals the Betti number $\beta_2$ plus the number $N$ of $2$--torsion summands of 
 $\Cohomol_1(_\Gamma \backslash X ; \thinspace \Z)$. 
 As $\Cohomol_0(_\Gamma \backslash X ; \thinspace \Z)$ contains no torsion, 
 we obtain that $\beta^1 = \beta_1 + N$. 
 The number $N$ vanishes for all absolute values of the discriminant less than $296$
 and has been determined in~\cite{higher_torsion} on a database
which includes all the Bianchi groups of ideal class numbers $1$, $2$, $3$ and $5$, 
most of the cases of ideal class number $4$,
as well as all of the cases of discriminant absolute value bounded by $500$.
\begin{notation} \label{notation:c}
Denote by $c$ the co-rank (i.e., the rank of the cokernel) of the map 
 \\ $\Cohomol^{1} (_\Gamma \backslash X; \thinspace \F_2) \rightarrow \Cohomol^{1} (_\Gamma \backslash X_s; \thinspace \F_2) $
 induced by the inclusion $X_s \subset X$.
\end{notation}

\begin{notation}
 Let $v$ denote the number of conjugacy classes of subgroups of quaternionic type~$\Q$ in $\SLO$,
 whether or not they are contained in a binary tetrahedral group $\Te$, and define 
\begin{center}$\sign(v) := $\scriptsize$\begin{cases} 
                   0, & v = 0,\\
		    1,& v> 0.
                  \end{cases}$\normalsize
\end{center}\end{notation}
There is a formula for $v$ in terms of the prime divisors of the discriminant of the ring 
$\calO_{-m}$ of integers 
\cite{Kraemer}. We use results from~\cites{RahmTorsion, AccessingFarrell} 
that the endpoints of $_\Gamma \backslash X_s$
are precisely the orbits of vertices with stabilizer group $\Te$,
and that the bifurcation points of $_\Gamma \backslash X_s$
are precisely the orbits of vertices with stabilizer group $\Q$.
We consider endpoints and bifurcation points of $_\Gamma \backslash X_s$
as ``necessary'' vertices,
because they cannot be eliminated during the reduction of the torsion subcomplex $X_s$.  This is 
due to the cohomology of their stabilizers,
which is different from the cohomology of all edge stabilizers.
The reduction of the torsion subcomplex $X_s$ eliminates all of the other vertices,
except for one orbit of vertices on components of type $(o)$, 
which is needed for the cell structure of $\circlegraph$, 
but can be chosen arbitrarily on the component~\cites{RahmTorsion, AccessingFarrell}.
This is why we can think of $v$ as the number of ``necessary'' vertices 
(endpoints or bifurcation points) of $_\Gamma \backslash X_s$.  Further, 
we can count these as the vertices of the $\Gamma$-quotient of a reduced non-central $2$-torsion subcomplex, 
in contrast to the ``spurious'' vertices found on connected components of type $\circlegraph$, which we omit. 
    
\begin{prop}\label{prop:relativeXs}
There is an isomorphism $E_2^{p,q}(X,X_s) \cong \Cohomol^p(_\Gamma \backslash X,\thinspace _\Gamma \backslash X_s)$.  In particular,
\[
\dim_{\ef} E_2^{p,q}(X,X_s) =
\begin{cases}
  0,                      &   p = 0,  \\
  \beta^1(_\Gamma \backslash X) + c + \chi(_\Gamma \backslash X_s) - 1,    &  p = 1,  \\
  \beta^2(_\Gamma \backslash X) +c,    &  p = 2.  
\end{cases}
\]
\end{prop}
 
 \begin{proof}
 We start with the split short exact sequence of chain complexes from Sequence~(\ref{cellularSES}), 
\begin{equation}\label{eqn:oddSES}
    0 \rightarrow E_1^{p,q}(X, X_s) \rightarrow  E_1^{p,q}(X) \rightarrow E_1^{p,q}(X_s) \rightarrow 0.
\end{equation}
Every cell stabilizer in $(X, X_s)$ has cohomology isomorphic to $\Cohomol^*(\Z/2)$, so
\[
  E_1^{p,q}(X, X_s)  \cong \bigoplus\limits_{\sigma \in  _\Gamma\backslash  (X,X_s)^{(p)}} \Cohomol^q( \Z/2 ).
\]
As $\Cohomol^q( \Z/2 ) \cong \ef$, which has trivial automorphism group,
the isomorphism is constant over all cells.  That is, the inclusion of an $n$-cell into an $(n+1)$-cell induces a unique isomorphism 
in the cohomology of the associated isotropy groups.  Hence $ E_1^{p,q}(X, X_s) \cong \ef \otimes_{\Z}  C^p(_\Gamma \backslash (X, X_s))$,
with the differential $d_1^{p,q}$ given as the coboundaries of 
$C^p(_\Gamma \backslash (X, X_s)) \cong C^p(_\Gamma \backslash X)/C^p(_\Gamma \backslash X_s)$.
This implies our first claim, that
\[
E_2^{p,q}(X,X_s) \cong   \Cohomol^p(_\Gamma \backslash X, \ _\Gamma \backslash X_s).
\]
As~$X$ is a $2$-dimensional cell complex and $X_s$ is $1$-dimensional, 
the long exact sequence associated to the relative cohomology of the pair 
$(_\Gamma \backslash X, \ _\Gamma \backslash X_s)$ is concentrated in
\begin{center} 
\begin{tikzpicture}[descr/.style={fill=white,inner sep=1.5pt}]
        \matrix (m) [
            matrix of math nodes,
            row sep=1em,
            column sep=2.5em,
            text height=1.99ex, text depth=0.75ex
        ]
        {  
	   & \Cohomol^2(_\Gamma \backslash X, \ _\Gamma \backslash X_s)  &\Cohomol^{2} (_\Gamma \backslash X)& 0 & \hdots \\ 
             & \Cohomol^1(_\Gamma \backslash X, \ _\Gamma \backslash X_s) &\Cohomol^{1} (_\Gamma \backslash X) & \Cohomol^{1} (_\Gamma \backslash X_s) &  \\
	 0 & \Cohomol^0(_\Gamma \backslash X, \ _\Gamma \backslash X_s) & \Cohomol^{0} (_\Gamma \backslash X) & \Cohomol^{0} (_\Gamma \backslash X_s) &\\ 
        };

        \path[overlay,->, font=\scriptsize,>=latex]
        (m-1-4) edge (m-1-5) 
        (m-1-2) edge (m-1-3) 
        (m-1-3) edge (m-1-4) 
        (m-2-4) edge[out=-355,in=-175]  (m-1-2)
        (m-2-2) edge (m-2-3)
        (m-2-3) edge (m-2-4)
        (m-3-4) edge[out=-355,in=-175]  (m-2-2)
        (m-3-2) edge (m-3-3)
        (m-3-3) edge (m-3-4)
        (m-3-1) edge (m-3-2)
;
\end{tikzpicture}
\end{center}
Since $X_s \subset X$ with $X$ connected, we immediately see that $\Cohomol^{0} (_\Gamma \backslash X)$ 
maps isomorphically to a 1-dimensional $\F_2$--subspace  in $\Cohomol^{0} (_\Gamma \backslash X_s)$,
yielding $\Cohomol^0(_\Gamma \backslash X, \ _\Gamma \backslash X_s) = 0$.
Therefore, the map from $\Cohomol^{0} (_\Gamma \backslash X_s)$ to $\Cohomol^1(_\Gamma \backslash X, \ _\Gamma \backslash X_s)$
has rank $\beta^0(_\Gamma \backslash X_s) -1$.
In addition, using the co-rank~$c$ from Notation~\ref{notation:c}, we have to complement $\Cohomol^1(_\Gamma \backslash X, \ _\Gamma \backslash X_s)$
by an $\F_2$--subspace of dimension 
$\beta^1(_\Gamma \backslash X)- (\beta^1(_\Gamma \backslash X_s) -c)$.
This yields the claimed formula $\chi(_\Gamma \backslash X_s) -1 +\beta^1(_\Gamma \backslash X) +c$ 
for the dimension of $\Cohomol^1(_\Gamma \backslash X, \ _\Gamma \backslash X_s)$.
The remaining terms of the long exact sequence produce 
\[
\Cohomol^2(_\Gamma \backslash X, \ _\Gamma \backslash X_s) \cong \Cohomol^{2} (_\Gamma \backslash X) \oplus ( \F_2)^c .
\]
\end{proof}

\begin{prop}\label{prop:relativeXsprime}
In the second row of the equivariant spectral sequence, \\
\[
\dim_{\ef} E_2^{p,2}(X,\xsp) =
\begin{cases}
  1 - \sign(v),                                   &   p = 0,  \\
  \beta^1(_\Gamma \backslash X) + v - \sign(v),   &  p = 1,  \\
  \beta^2(_\Gamma \backslash X),                  &  p = 2.  
  \end{cases}
\]
\end{prop}
 
\begin{proof}
We note that from Proposition~\ref{prop:SLfincoh},
 the cohomology rings of stabilizers of cells in $(X, \xsp)$ 
 are either isomorphic to $\Cohomol^*(\Z/4)$ or $\Cohomol^*(\Z/2)$. 
 Both cohomology rings have period evenly dividing~$2$.  
 Consequently, $E_1^{p,q}(X, \xsp)$ and $E_2^{p,q}(X, \xsp)$ are $2$-periodic in $q$, 
 and it is enough to investigate $E_1^{p,0}(X, \xsp)$ and $E_2^{p,0}(X, \xsp)$.   
We consider the short exact sequence of chain complexes defining~$ E_1^{p,0}(X, \xsp)$,
\[
 0 \rightarrow E_1^{p,0}(X, \xsp) \rightarrow  E_1^{p,0}(X) \rightarrow E_1^{p,0}(\xsp) \rightarrow 0.
\]
Since $ E_1^{0,0}(\xsp) \cong (\ef)^v$ and $ E_1^{p,0}(\xsp) = 0$ for $p \geq 1$,
this sequence of chain complexes is concentrated in the following diagram.

\begin{equation*} 
 \begin{tikzpicture}[descr/.style={fill=white,inner sep=1.5pt}]
        \matrix (m) [
            matrix of math nodes,
            row sep=1.4em,
            column sep=1.4em,
            text height=3.95ex, text depth=2.95ex
        ]
        {{} & 0   & 0 & 0  & {} \\
        0 &  E_1^{2,0}(X, \xsp) &  E_1^{2,0}(X) & 0  & 0  \\        
        0 &  E_1^{1,0}(X, \xsp) &  E_1^{1,0}(X) & 0  & 0  \\    
        0 &  E_1^{0,0}(X, \xsp) &  E_1^{0,0}(X)  & (\ef)^v & 0  \\    
         & 0 & 0 & 0  & \\
        };
        \path[overlay,->, font=\scriptsize,>=latex]
        (m-2-2) edge node[descr,xshift=-4.4ex] {$d_1^{2,0}$}  (m-1-2)
        (m-2-3) edge node[descr,xshift=-4.4ex] {$d_1^{2,0}$}  (m-1-3)
        (m-2-4) edge (m-1-4)
        (m-2-1) edge (m-2-2)
	(m-2-2) edge (m-2-3) 
        (m-2-3) edge (m-2-4)
        (m-2-4) edge (m-2-5)
	(m-3-2) edge node[descr,xshift=-4.4ex] {$d_1^{1,0}$} (m-2-2)
        (m-3-3) edge node[descr,xshift=-4.4ex]   {$d_1^{1,0}$}  (m-2-3)
        (m-3-4) edge  (m-2-4) 
        (m-3-1) edge (m-3-2)
	(m-3-2) edge (m-3-3) 
        (m-3-3) edge (m-3-4)
        (m-3-4) edge (m-3-5)
	(m-4-2) edge node[descr,xshift=-4.4ex]  {$d_1^{0,0}$} (m-3-2)
        (m-4-3) edge node[descr,xshift=-4.4ex]    {$d_1^{0,0}$} (m-3-3)
        (m-4-4) edge (m-3-4) 
        (m-4-1) edge (m-4-2)
	(m-4-2) edge (m-4-3) 
        (m-4-3) edge (m-4-4)
        (m-4-4) edge (m-4-5)
	(m-5-2) edge (m-4-2)
        (m-5-3) edge  (m-4-3)
        (m-5-4) edge (m-4-4) 
  ;
\end{tikzpicture} 
\end{equation*}\normalsize

Since $E_2^{p,0}(X) \cong \Cohomol^p(_\Gamma \backslash X)$ by  Theorem~\ref{thm:E2page},  
the long exact sequence obtained with the snake lemma from the above diagram is 
\begin{center} 
\begin{tikzpicture}[descr/.style={fill=white,inner sep=1.5pt}]
        \matrix (m) [
            matrix of math nodes,
            row sep=1em,
            column sep=2.5em,
            text height=1.99ex, text depth=0.75ex
        ]
        {    &   E_2^{2,0}(X, \xsp) &  \Cohomol^2(_\Gamma \backslash X) & 0 \\
             &   E_2^{1,0}(X, \xsp) &  \Cohomol^1(_\Gamma \backslash X) & 0 \\
	   0 &   E_2^{0,0}(X, \xsp) &  \Cohomol^0(_\Gamma \backslash X) & (\ef)^v  \\  
        };

        \path[overlay,->, font=\scriptsize,>=latex]
        (m-1-2) edge (m-1-3) 
        (m-1-3) edge (m-1-4) 
        (m-2-4) edge[out=-355,in=-175]  (m-1-2)
        (m-2-2) edge (m-2-3)
        (m-2-3) edge (m-2-4)
        (m-3-4) edge[out=-355,in=-175]  (m-2-2)
        (m-3-2) edge (m-3-3)
        (m-3-3) edge (m-3-4)
        (m-3-1) edge (m-3-2)
;
\end{tikzpicture}
\end{center}
By the $2$-periodicity, the top row of the sequence already yields one case of the claimed formula,
$$\dim_{\ef} E_2^{2,2}(X,\xsp) = \dim_{\ef} E_2^{2,0}(X,\xsp) = \beta^2(_\Gamma \backslash X).$$
It remains to study the five-term exact sequence given by the two bottom rows. 
Here, we note that $\Cohomol^0(_\Gamma \backslash X) \cong \ef$, 
and that $E_2^{0,0}(X, \xsp)$ is isomorphic to the kernel of the map
\[
\Cohomol^0(_\Gamma \backslash X) \rightarrow (\ef)^v,
\]
which is non-trivial precisely when $v > 0$.  
Therefore, $\dim_{\ef} E_2^{0,0}(X,\xsp) = 1 - \sign(v)$.  
Thus, we can extract a short exact sequence
\[
0 \rightarrow (\ef)^{v - \sign(v)}  \ {
\rightarrow} \  E_2^{1,0}(X, \xsp)  \rightarrow  \Cohomol^1(_\Gamma \backslash X) \rightarrow 0,  
\]
which implies $\dim_{\ef} E_2^{1,0}(X,\xsp) \cong \beta^1(_\Gamma \backslash X) + v - \sign(v)$.  
\end{proof}
 
 \begin{corollary}\label{cor:E2page}
 The $E_2$ page of the equivariant spectral sequence with $\F_2$--coefficients
 associated to the action of $\Gamma$ on $X$ is concentrated in the columns $n \in \{0, 1, 2\}$
and has the following form:   
 
 \[  
\begin{array}{l | cccl}
q = 4k+3 &  E_2^{0,3}(X_s)     &   E_2^{1,3}(X_s) \oplus (\ef)^{a_1}  &   (\ef)^{a_2} \\
q = 4k+2 &  \Cohomol^2_\Gamma(\xsp) \oplus (\ef)^{1 -\sign(v)}     &    (\ef)^{a_3}&   \Cohomol^2(_\Gamma \backslash X)  \\
q = 4k+1 &  E_2^{0,1}(X_s)     &  E_2^{1,1}(X_s) \oplus  (\ef)^{a_1} &   (\ef)^{a_2} \\
q = 4k   & \F_2  &   \Cohomol^1(_\Gamma \backslash X) &  \Cohomol^2(_\Gamma \backslash X) \\
\hline k \in \mathbb{N} \cup \{0\} &  n = 0 & n = 1 &  n = 2
\end{array}
\]
where 
\[\begin{array}{ll}
a_1 & =  \beta^1(_\Gamma \backslash X) + c  +  \chi(_\Gamma \backslash X_s) -1   \\
a_2 & =  \beta^{2} (_\Gamma \backslash X) +c, \\
a_3 & =  \beta^{1} (_\Gamma \backslash X) +v -\sign(v).
\end{array}
\]   

\end{corollary}
 
\begin{proof}
Theorem~\ref{thm:E2page} describes the structure of the $E_2$ page of the spectral sequence by row as cohomology of the quotient space ($q = 4k$) and direct sums of cohomology groups of subcomplexes and relative complexes (other values of $q$).  Propositions \ref{prop:relativeXs} and \ref{prop:relativeXsprime} provide a description of the cohomology of the relative complexes.
\end{proof}


\section{Calculation of the spectral sequence on the subcomplex} \label{Calculation of the non-polynomial part of the cohomology}

In Corollary~\ref{cor:E2page}, we have expressed the $E_2$ page of the equivariant spectral sequence converging to
$\Cohomol^*(\SLO)$
in terms of invariants of the quotient space and the $E_2$ page associated to the non-central 2-torsion subcomplex $X_s$. 
In this section, we will investigate that latter $E_2$ page, $E_2^{p,q}(X_s)$.
Recall that the only $2$-torsion elements which stabilize cells outside of the non-central $2$--torsion subcomplex are in the 
$\Z/2$ center. This allows us to relate back to and use results about the $2$--torsion subcomplex for $\PSLtwo(\ringOm)$. 
What we have to do in order to establish this relation is to show that $E_2^{p,q}(X_s)$ 
splits as a direct sum indexed by the connected components of $X_s$;
then further that the reduction of subcomplex components is in complete agreement with the analogous reduction for $\PSLtwo(\ringOm)$.
We do this in the following lemma.

\begin{lemma} \label{splitting over components}
The terms $E_2^{p,q}(X_s)$ split into direct summands each with support on one connected component
of the quotient of a \emph{reduced} non-central $2$--torsion subcomplex.
\end{lemma}

\begin{proof}
An argument in Section 6 of~\cite{AccessingFarrell} explains why the mod $2$ cohomology of $\PSLtwo(\ringOm)$
 splits into a direct sum above the virtual cohomological dimension.
The argument references the calculations in~\cite{BerkoveMod2},
and notes that classes that arise in one component of the non-central $2$--torsion subcomplex do not restrict to subgroups in others. 
This implies that products between classes that come from distinct components multiply trivially in cohomology.

We need to extend this result from the projective special linear group to the special linear group.
In other words, we need to make sure that the result is compatible with the central extension of $\PSLtwo(\ringOm)$ 
by $\{\pm 1\} \cong \Z/2$.
Summarizing results in~\cite{RahmTorsion} where the $2$--torsion subcomplex is developed, we note 
that all edge fusions, which happen during the reduction of the subcomplex for $\PSLO$,
 remove a vertex with stabilizer
 $\Z/2$ or $S_3$.  The adjacent edges which are fused both have stabilizer $\Z/2$.
  Since $\Cohomol^*(S_3) \cong \Cohomol^*(\Z/2)$ and since this is an isomorphism of rings,
 the mod $2$ cohomology of the component corresponding to the 
 $2$--torsion subcomplex is ring-isomorphic
 to the mod $2$ cohomology coming from the original component.  

Now the central $\{\pm 1\}$ group acts trivially on the retracted cell complex,
 so as noted in the discussion after Definition~\ref{non-central torsion subcomplex}, 
the \emph{non-central} $2$--torsion subcomplex for $\SLtwo(\ringOm)$
 is identical to the $2$--torsion subcomplex for $\PSLtwo(\ringOm)$. 
 The stabilizers for $\SLtwo(\ringOm)$ are extensions of the stabilizers in $\PSLtwo(\ringOm)$ by $\{\pm 1\}$.
  In particular,
$\Z/2$ is extended to $\Z/4$ and $S_3$ is extended to~$\Di$.
Since there is a ring isomorphism  $\Cohomol^*(\Di) \cong \Cohomol^*(\Z/4)$,
we conclude that at vertices where there was an edge fusion in the $2$--torsion subcomplex in the $\PSLtwo(\ringOm)$ case,
 there will also be an edge fusion in the non-central $2$--torsion subcomplex in the $\SLtwo(\ringOm)$ case.
 Therefore, the non-central $2$--torsion subcomplex can be reduced as in the $\PSLtwo(\ringOm)$ case presented in~\cite{RahmTorsion}.
This yields a splitting for all of the terms $E_2^{p,q}(X_s)$.
\end{proof}

The main task in the remainder of this section will be to use the cohomology and restriction information gathered in Section~\ref{sec:pieces}
in order to determine the equivariant cohomology supported on the individual connected components of the quotient of a reduced non-central 
$2$--torsion subcomplex.  
The key observation for this task is that the individual connected components of the quotient of a reduced non-central 
$2$--torsion subcomplex correspond to groups which can be described as amalgamated products and HNN extensions.  
We recall the definition of an HNN extension.    

\begin{defn}
 Let $\varphi:G_2 \rightarrow G_1$ be an injection of $G_2$ into $G_1$.   An  \textit{HNN extension of $G_1$} is a group with presentation
\[
G_1 \ast_{G_2} = \  < t,G_1 \mid t^{-1}gt = \varphi(g), g \in G_2 >.
\]
The element $t$ is often referred to as the \textit{free letter}.  
\end{defn}
We note that the notation $G_1 \ast_{G_2}$ is not completely descriptive, since there may be many possible injective maps $\varphi$.  

Via Bass--Serre theory, it is known that amalgamated products and HNN extensions both act on trees.  In the amalgamated product $G=G_1 \ast_H G_2$, there is an action where the fundamental domain is given by two vertices with stabilizers $G_1$ and $G_2$,
 connected by an edge with stabilizer~$H$.  In the HNN extension $G=G_1 \ast_{G_2}$, the fundamental domain is a single vertex with a loop where the vertex stabilizer is $G_1$ and the edge stabilizer is $G_2$.  

We can calculate the cohomology of HNN extensions and amalgamated products using the next result.   

\begin{thm}
\label{thm:quotientLES}
Let $\sigma_0$ and $\sigma_1$  be $0$--cells and $1$--cells in a fundamental domain, and let $G_v$ and $G_e$ be the vertex and edge stabilizers respectively.  Then there is a long exact sequence in cohomology
\begin{equation}
\cdots \ararrow  \bigoplus_{e \in {\sigma}_1 } \Cohomol^{i-1}(G_e)   \label{eqn:TreeLES}
          \drarrow \Cohomol^i (G) \brarrow \bigoplus_{v \in {\sigma}_0 } \Cohomol^i (G_v)  
          \ararrow  \bigoplus_{e \in {\sigma}_1 } \Cohomol^i (G_e) \drarrow \cdots    
\end{equation}
The direct sum is over one edge and two vertices if $G$ is an amalgamated product
(resulting in a Mayer--Vietoris sequence),
 and over one edge and one vertex if $G$ is an HNN extension.
\end{thm}

The following is another result from Bass--Serre theory.

\begin{prop}
\label{map-structures}
In the long exact sequence of Theorem~\ref{thm:quotientLES}, the map $\beta$ is the restriction map.  In an amalgamated product, $\alpha = res^{G_1}_H - 
res^{G_2}_H$.  
In an HNN extension, $\alpha = res^{G_1}_{G_2} - \varphi^*$, 
where~$\varphi^*$ is the map induced by conjugating $G_2$ by the free letter.
\end{prop}

In fact, when the equivariant spectral  sequence only has two non-zero columns,
the $E_1$ page degenerates into a Wang sequence which is precisely the long exact sequence~(\ref{eqn:TreeLES})
in Theorem~\ref{thm:quotientLES}\thinspace{}.  Furthermore, the map $\alpha$ is the $d_1$ differential.
We note that although we are primarily interested in an additive calculation of the cohomology, the long exact sequence can also be used to determine ring information,
as $\ker(\alpha) = \im(\beta)$ and this is a ring isomorphism.

We will use the equivariant spectral sequence to carry out our cohomology calculations.  
We recall that Remark~\ref{rem:separatepolyext} states that concerning the occurring restriction maps on cohomology, 
only reduced classes restrict non-trivially to reduced classes. 

\begin{lemo} \label{lem:lemOne}
Let $G \cong \Z/4 \ast_{\Z/4}$.  Then
$\dim_{\ef} \Cohomol^q(G) = $\scriptsize$
  \begin{cases}
    1,          &    q = 0; \\
    2,          &    q \geq 1. \\ 
  \end{cases}
$
\end{lemo}

\begin{proof}
The quotient of the tree acted on by $G$ has shape $\circlegraph$,
with vertex and edge stabilizers both isomorphic to $\Z/4$.  Using the restriction maps from Proposition~\ref {prop:restriction} we set up the $E_1$ and $E_2$ pages of the equivariant spectral sequence for this HNN extension below.  
\vspace{.5em} \\
\begin{center}
\begin{tabular}{ccc}

$E_1$ page &  & $E_2$ page \\
\small 
$
\begin{array}{l | clcl}
4 & \ef &  \longrightarrow & \ef  \\
3 & \ef &  \longrightarrow & \ef  \\
2 & \ef &  \longrightarrow & \ef  \\
1 & \ef &  \longrightarrow & \ef  \\
0 & \ef &  \longrightarrow & \ef  \\
\hline &  0 & &  1
\end{array}
$
& $\buildrel\ {\Homol_*(d_1)} \over \Longrightarrow$ &
\small
$
\begin{array}{l | clcl}
4 & \ef &  & \ef  \\
3 & \ef                & & \ef  \\
2 & \ef &   & \ef  \\
1 & \ef                &   & \ef  \\
0 & \ef &   &  \ef \\
\hline &  0 & &  1
\end{array}
$
\end{tabular}\normalsize
\end{center}

The arrows in the spectral sequence are the $d_1$ differentials given by 
$d_1 = res^{\Z/4}_{\Z/4} - \varphi^*$.  
Since both maps $res^{\Z/4}_{\Z/4}$ and $\varphi^*$
have identical effects on $\Cohomol^*(\Z/4$), we obtain $d_1 = 0$.  
So the $E_1$ page of the spectral sequence is the same as the $E_\infty$ page.
\end{proof}


\begin{lemiota} \label{lem:iota} Let $G \cong \Te \ast_{\Z/4} \Te$.   Then
$ \dim_{\ef}\Homol^q(G) = $\scriptsize$
  \begin{cases}
       1,          &    q = 0;\\
       2,          &    q = 4k, \ k \geq 1;  \\ 
       0,          &    q =4k+1;  \\
       1,          &    q = 4k+2; \\
       3,          &    q=4k+3.
  \end{cases}
$
\end{lemiota}

\begin{proof}
The quotient of the tree acted on by $G$ has shape $\edgegraph$.  Both vertex stabilizers are isomorphic to $\Te$, and the edge stabilizer is isomorphic to $\Z/4$. Propositions~\ref{prop:SLfincoh} and~\ref {prop:restriction} allow us to build the $E_1$ and $E_2$ pages of the equivariant spectral sequence for this amalgamated product.  Both pages are below.
\vspace{.5em} \\
\begin{center}
\begin{tabular}{ccc}

$E_1$ page &  & $E_2$ page \\
\small 
$
\begin{array}{l | clcl}
6 &  0             &  \longrightarrow & \ef  \\
5 &  0             &  \longrightarrow & \ef  \\
4 & (\ef)^2 &  \longrightarrow & \ef  \\
3 & (\ef)^2 &  \longrightarrow & \ef  \\
2 &  0            &  \longrightarrow & \ef  \\
1 &  0            &  \longrightarrow & \ef  \\
0 & (\ef)^2 &  \longrightarrow & \ef  \\
\hline &  0 & &  1
\end{array}
$
& $\buildrel\ {\Homol_*(d_1)} \over \Longrightarrow$ & \small
$
\begin{array}{l | clcl}
6 &                          &  &  \ef  \\
5 &                          &  &  \ef  \\
4 & \ef                    &  &           \\
3 & (\ef)^2             &  & \ef  \\
2 &                          &  & \ef  \\
1 &                          &  & \ef  \\
0 & \ef    &  &   \\
\hline &  0 & &  1
\end{array}
$
\end{tabular}\normalsize
\end{center}
On the $E_1$ page, the only non-zero restriction maps are the ones on reduced classes in rows $4k$.
\end{proof}


We consider the group $G \cong \left( \left( \Q \ast_{\Z/4} \Q \right) \ast_{\Z/4} \right) \ast_{\Z/4}$, constructed as follows.
The copies of~$\Q$ are amalgamated over a copy of $\Z/4$.  
The group $G$ is formed from $\Q \ast_{\Z/4} \Q$ via an iterated HNN extension.  
The first HNN extension takes a second $\Z/4$ subgroup, non-conjugate to the first, in one copy of $\Q$ to a non-conjugate $\Z/4$ subgroup in the other.  
The second HNN extension is defined similarly using the third non-conjugate $\Z/4$ subgroups.  
\begin{lemtheta} \label{lem:theta}
Let $G \cong \left( \left( \Q \ast_{\Z/4} \Q \right) \ast_{\Z/4} \right) \ast_{\Z/4}$ with HNN extensions as specified above.  
Then 
\begin{center}
$ \dim_{\ef}\Homol^q(G) = $\scriptsize$
  \begin{cases}
    1,          &    q = 0;\\
    4,          &    q = 4k, \   k \geq 1;  \\ 
    4,          &    q =4k+1;  \\
    5,          &    q = 4k+2; \\
    5,          &    q=4k+3.
  \end{cases}
$
\end{center}
\end{lemtheta}

\begin{proof}
The quotient of the tree acted on by $G$ has shape $\graphFive$.  
The two vertex stabilizers are isomorphic to $\Q$, and the three non-conjugate edge stabilizers are all isomorphic to $\Z/4$. 

The $E_1$ page of the spectral sequence has two copies of $\Cohomol^*(\Q)$ in column $0$ 
and three copies of $\Cohomol^*(\Z/4)$ in column $1$.  
Our main concern is to determine the action of the $d_1$ differential from 
$E_1^{0,1}$ to $E_1^{1,1}$.  
By Proposition~\ref {prop:restriction}, we know that in dimension $1$, 
a non-zero restriction $res^{\Q}_{\Z/4}$  can be traced back to a non-zero restriction 
$res^{\Dtwo}_{\Z/2}$.  
In cohomology, let $b_{11}, b_{12}$, and $b_{13}$ denote the exterior classes in $\oplus_3 \Cohomol^1(\Z/4)$ 
corresponding to the three edge stabilizers,
 and let $x_{11}, y_{11}, x_{12}, y_{12}$ be the nilpotent classes in $\Cohomol^1(\Q) \oplus \Cohomol^1(\Q)$
 corresponding to the two vertex stabilizers.
  Further, assume that $x_{11}$ and $x_{12}$ both restrict isomorphically to $b_{11}$
 via the correspondence from Proposition~\ref {prop:restriction}, and that $y_{11}$ and $y_{12}$ both restrict isomorphically to $b_{12}$. 
 In $\Dtwo$, the product of any two non-zero elements is the final non-zero element. 
 Consequently, in cohomology, $x_{11}, y_{11}, x_{12}$ and $y_{12}$ all restrict to $b_{13}$.  That is, 
\[
\begin{array}{lcl}
d_1^{0,1}(x_{11}) & = & b_{11} + b_{13}  \\
d_1^{0,1}(y_{11}) & = & b_{12} + b_{13}  \\
d_1^{0,1}(x_{12}) & = & b_{11} + b_{13}  \\
d_1^{0,1}(y_{12}) & = & b_{12} + b_{13}.  
\end{array}
\]
So $d_1^{0,1}$ has rank $2$.  The only other non-zero differential is between polynomial classes in rows $4k$.   Let $e_{41}$ and $e_{42}$ be the polynomial 
generators of $\Cohomol^4(\Q) \oplus \Cohomol^4(\Q)$, and let $e_{21}^2$, $e_{22}^2$, $e_{23}^2$ be the squares of the two-dimensional polynomial generators in $\oplus_3 \Cohomol^4(\Z/4)$.  
Based on the geometry of the quotient of the tree acted on by $G$, we have
\[ 
  d_1(e_{41}) = d_1(e_{42})= e_{21}^2 + e_{22}^2 + e_{23}^2.  
\]
This implies that both the kernel and image of $d_1$ are $1$--dimensional.  We can now completely determine the $E_1$ and $E_2$ pages of the equivariant spectral sequence.
\vspace{.5em} \\
\begin{center}
\begin{tabular}{ccc}

$E_1$ page &  & $E_2$ page \\
\small 
$
\begin{array}{l | clcl}
6 & (\ef)^4  &  \longrightarrow & (\ef)^3  \\
5 & (\ef)^4  &  \longrightarrow & (\ef)^3  \\
4 & (\ef)^2  &  \longrightarrow & (\ef)^3  \\
3 & (\ef)^2  &  \longrightarrow & (\ef)^3  \\
2 & (\ef)^4  &  \longrightarrow & (\ef)^3  \\
1 & (\ef)^4  &  \longrightarrow & (\ef)^3  \\
0 & (\ef)^2  &  \longrightarrow & (\ef)^3  \\
\hline &  0 & &  1
\end{array}
$
& $\buildrel\ {\Homol_*(d_1)} \over \Longrightarrow$ &\small
$
\begin{array}{l | clcl}
6 & (\ef)^4           &  &  (\ef)^3    \\
5 & (\ef)^2           &  &  \ef           \\
4 & \ef                  &  &  (\ef)^2    \\
3 & (\ef)^2           &  &  (\ef)^3    \\
2 & (\ef)^4           &  &  (\ef)^3    \\
1 & (\ef)^2           &  &  \ef           \\
0 & \ef                  &  &  (\ef)^2    \\
\hline &  0 & &  1
\end{array}
$
\end{tabular}\normalsize
\end{center}

\end{proof}


\vbox{
\begin{lemrho} \label{lem:rho}
Let $G \cong (\Q \ast_{\Z/4}) \ast_{\Z/4} \Te$, where the HNN extension identifies two non-conjugate copies of $\Z/4$ in $\Q$; 
and the third conjugacy class of $\Z/4$ in $\Q$ is the amalgamated subgroup with $\Te$.  
Then 
\begin{center}
$ \dim_{\ef}\Homol^q(G) = $\scriptsize$
  \begin{cases}
    1                    &    q = 0 \\
    3,                  &    q=  4k, \  \geq 1; \\
    2,                  &    q = 4k +1;\\
    3,                  &    q =4k+ 2;  \\
    4,                 &    q = 4k+3. \\
  \end{cases}
$
\end{center}
\end{lemrho}
}

\begin{proof}
The quotient of the tree acted on by $G$ has shape $\graphTwo$.  The vertex stabilizer incident to the loop is isomorphic to $\Q$.  The other vertex stabilizer is isomorphic to $\Te$, and the two edge stabilizers are isomorphic to $\Z/4$. 
This is the most complicated case, since the fundamental domain for $G$ consists of two vertices and two edges, and the restriction maps arising from the edges have to be taken into account.  

We follow the approach used for a similar mod $2$ calculation in Lemma 3.1 of~\cite{BerkoveMod2}. 
We start with an analysis of the component of $d_1: \Cohomol^*(\Q) \rightarrow \Cohomol^*(\Z/4)$,
where the copy of $\Z/4$ is the stabilizer of the loop.
 Referring back to the long exact sequence given in Equation~\ref{eqn:TreeLES}, 
this component of $d_1$  is the difference  $res^{\Q}_{\Z/4} - \varphi^*$, 
where the restriction map is induced by subgroup injection on one side of the loop,
and by the twisting from the HNN extension on the other one.   
We tweak the notation for classes from Proposition~\ref{prop:SLfincoh} by denoting by $b_{11} \in \Cohomol^1(\Z/4)$ 
the class in the loop stabilizer and $b_{12} \in \Cohomol^1(\Z/4)$ the class from the edge stabilizer.  
(We use the same convention for reduced classes.)   
By the mod $2$ calculation in~\cite{BerkoveMod2},
which carefully tracks the effect of the twisting map on cohomology, 
one can show that the generators $x_1$ and $y_1$ mentioned in Proposition~\ref{prop:SLfincoh} have images
\[
res^{\Q}_{\Z/4}(x_1) = b_{11}, \text{ but } res^{\Q}_{\Z/4}(y_1) =  0.
\] 
That is, the restriction only detects the subgroup on one side of the HNN extension.  
The map on cohomology induced by $\varphi$, on the other hand, detects the subgroup on both sides:
\[
\varphi^*(x_1) = \varphi^*(y_1) = b_{11}.
\]
Consequently, the component of $d_1$ mapping to the loop sends $x_1$ to $0$ and $y_1$ to $b_{11}$.  
In addition, as $res^{\Q}_{\Z/4}(e_4) =   \varphi^*(e_4) = e_{21}^2$,
the same component of $d_1$ sends $e_4$ to $0$.  
The map $res^{\Q}_{\Z/4}$ to the unlooped edge with $\Z/4$ stabilizer is given by Proposition~\ref {prop:restriction},
 although we need to determine whether the nilpotent class with a nontrivial restriction is $x_1$, $y_1$, or $x_1 + y_1$. 
 In fact, a basis for $\Cohomol^*(\Q)$ can be chosen so that $res^{\Q}_{\Z/4}$ sends $x_1$ to $0$ and $y_1$ to $b_{12}$.  Summarizing, 
\vspace{.5em}
\[
  \begin{array}{lll} 
     res^{\Q}_{\Z/4 \oplus \Z/4} (y_1) &  =&   b_{11} + b_{12} \\
     res^{\Q}_{\Z/4 \oplus \Z/4} (e_4) &  = &  e_{22}^2
  \end{array}
\]
and all other classes are sent to $0$.  The rest of $d_1$ is given by the restriction map $res^{\Te}_{\Z/4}$, 
and this map is nontrivial only on the reduced class.  We have now completely determined the $d_1$ differential, so we can write down the $E_1$ and $E_2$ pages.  
\vspace{.5em} \\
\begin{center}
\begin{tabular}{ccc}

$E_1$ page &  & $E_2$ page \\
\small 
$
\begin{array}{l | clcl}
6 & (\ef)^2  &  \longrightarrow & (\ef)^2  \\
5 & (\ef)^2  &  \longrightarrow & (\ef)^2   \\
4 & (\ef)^2  &  \longrightarrow & (\ef)^2   \\
3 & (\ef)^2  &  \longrightarrow & (\ef)^2   \\
2 & (\ef)^2  &  \longrightarrow & (\ef)^2   \\
1 & (\ef)^2  &  \longrightarrow & (\ef)^2   \\
0 & (\ef)^2  &  \longrightarrow & (\ef)^2   \\
\hline &  0 & &  1
\end{array}
$
& $\buildrel\ {\Homol_*(d_1)} \over \Longrightarrow$ &
\small
$
\begin{array}{l | clcl}
6 & (\ef)^2           &  &  (\ef)^2   \\
5 & \ef                  &  &   \ef          \\
4 & \ef                  &  &  \ef           \\
3 & (\ef)^2           &  &  (\ef)^2   \\
2 & (\ef)^2           &  &  (\ef)^2   \\
1 &  \ef                 &  &   \ef         \\
0 & \ef                  &  &  \ef          \\
\hline &  0             & &  1
\end{array}
$
\end{tabular}
\end{center}
\normalsize
\end{proof}

\begin{note} \label{sub-sub}
As the sub-subcomplex $\xsp$ is $0$-dimensional, we can read off the equivariant cohomology $\Cohomol^\Gamma_q(\xsp)$
as the $E_1^{0,q}$-terms of its spectral sequence. 
This furthermore splits as a direct sum over the connected components of $X_s$,
and we write $\edgegraph '$,   $\graphFive '$, respectively $\graphTwo '$ for subsets of $\xsp$
with orbit space constituting the vertices of a connected component 
$\edgegraph$,   $\graphFive$, respectively $\graphTwo$ of $_\Gamma \backslash X_s$.
Then using Proposition~\ref{prop:SLfincoh}, we obtain on the known types of connected components
\begin{center}
 $\Cohomol^\Gamma_2(\edgegraph ') = 0$,    $\Cohomol^\Gamma_2(\graphFive ') \cong (\F_2)^4$,
  $\Cohomol^\Gamma_2(\graphTwo ') \cong (\F_2)^2$.
\end{center}
The set $\circlegraph '$ is empty, but we can abuse notation and also write $\Cohomol^\Gamma_2(\circlegraph ') = 0$.
\end{note}


\section{Decomposition of the second page differential} \label{The second page differential on the components of the torsion subcomplex}

Looking back at the four groups associated to non-central $2$--torsion subcomplex quotient components types shown in Table~\ref{table:subcomplexes}, 
we see that the groups associated to the quotient types 
$\edgegraph$, $\graphFive$, and $\graphTwo$ have periodic cohomology of period $4$ starting above degree $d = 0$,
 where $\alpha$ is the unique polynomial class in degree $4$ which is detected on the fiber of the extension.
In contrast, for the quotient type $\circlegraph$, the periodic cohomology has period $2$,
where the periodicity generator is again the unique polynomial class in dimension $2$.

\textit{In the equivariant spectral sequence for the action of $\SLtwo(\ringOm)$ on $X$, which we will denote by} E3SL, we have completely determined the $d_1$ differentials.  The final goal is to determine as much as possible about the $d_2$ differentials.  By periodicity, once we know what happens in the first four rows,  we know what happens in the entire spectral sequence. Our main technique involves an analysis of Steenrod operations in the E3SL.  The target of the $d_2$ differential is the second column of the E3SL, where all $2$-cell stabilizers are of type $\Z/2$.  We note that the Steenrod algebra on $\Cohomol^*(\Z/2)$ is generated by the operation $Sq^1(x_1) = x_1^2$.  

\vbox{
\begin{lemma} \label{Lem:EvenDegreeTrivial}
All classes in $E_2^{0,2q}$ of the {\rm E3SL} are $d_2$-cocycles.   
\end{lemma}

\begin{proof}
From Proposition~\ref{prop:SLfincoh}, we see that $Sq^1$ in the vertical edge of the E3SL is trivial in even degrees.  
The only classes that might need to be checked are in $\Cohomol^2(\Q)$, and the result follows since these classes arise from squares on the horizontal edge of the Lyndon--Hochschild--Serre spectral sequence associated to the extension  
\[
  1 \rightarrow  \Z/2 \rightarrow \Q \rightarrow \Dtwo \rightarrow 1.
\]

Going back to the E3SL, a non-zero target of the $d_2$ differential would be an odd-dimensional class in $\Cohomol^*(\Z/2)$ in the second column which has a non-trivial $Sq^1$.  However, since $Sq^1 d_2 = d_2 Sq^1$ this is impossible.   
\end{proof}
}

As our cell complex $X$ is 2-dimensional, we obtain the following corollary from this lemma and Corollary~\ref{cor:E2page}.

\begin{corollary}\label{cor:polnvanish}
In the equivariant spectral sequence associated to  $\SLtwo(\ringOm)$, the $d_2$ differential can only be nontrivial on $ E_2^{0,2q+1}(X)  \cong  E_2^{0,2q+1}(X_s)$  for $q \geq 0$.
\end{corollary}

We note that the $d_2^{0,q}$ 
differentials are block-diagonalizable with matrix blocks supported each by one connected component
of the quotient of the non-central $2$--torsion subcomplex.  
Working one component type at a time, we next show that in the equivariant spectral sequence, 
for components of type $\circlegraph$,  $d_2^{0,1}$ is trivial if and only if $d_2^{0,3}$ is.  
We will also show that the $d_2$ differential is trivial on components of type 
$\edgegraph$ and $\graphFive$, and $\graphTwo$ in degrees $q \equiv 3 \bmod 4$.   
This implies a vanishing result on components of type $\edgegraph$.

\begin{lemo} \label{lem:CircleExt}
The $d_2$ differential is nontrivial on cohomology on components of type $\circlegraph$ in degrees $q \equiv 1 \bmod 4$
if and only if it is nontrivial on these components in degrees $q \equiv 3 \bmod 4$.
\end{lemo}

\begin{proof}
The vertex stabilizers for components of type $\circlegraph$ are $\Z/4$, and $Sq^2 \left( \Cohomol^3(\Z/4) \right) \neq 0$.  Similarly, classes in $E_2^{2,2}$ are generated by $\Cohomol^2(\Z/2)$, and $Sq^2$ of these classes are also non-zero.  The result follows by applying periodicity, and noting that $d_2 Sq^2 = Sq^2 d_2$.
\end{proof}

\begin{lemiota} \label{iota-case}
The $d_2^{\rm E3SL}$ differential is trivial on connected components of type $\edgegraph$.
\end{lemiota}

\begin{proof}
By Lemma \ref{Lem:EvenDegreeTrivial}, it is sufficient to restrict ourselves to classes in odd degree.  For $d_2^{0,q}$ with $q \equiv 3 \bmod 4$, from the description of the classes from Lemma \ref{lem:iota} and the calculation of $\Cohomol^*(\Te)$ in Proposition~\ref{prop:SLfincoh}, all Steenrod squares on $b_3$ are trivial.  However, in the second column of the E3SL, $Sq^2$ is non-trivial on $E_2^{2,2}$.  By the compatibility of the Steenrod operations with the $d_2$ differential, $d_2$ must be the zero map.  
\end{proof}

\begin{lemthetarho}\label{lem:gFsquares} \label{theta-case}
The $d_2$ differential is trivial on components of types $\graphTwo$ and $\graphFive$ in dimensions $q \equiv 3 \bmod 4$.
\end{lemthetarho}

\begin{proof}
The proof is identical to that in Lemma \ref{iota-case}, since, by Proposition~\ref{prop:SLfincoh}, $Sq^2$ is trivial on $\Cohomol^3(\Q)$ as well as $\Cohomol^3(\Te)$. 
\end{proof}

\begin{remark} \label{36}
It remains to determine how $d_2$ behaves on components on types $\graphTwo$ and $\graphFive$ in degrees $q \equiv 1 \bmod 4$, which is equivalent to understanding 
$d_2^{0,1}: E_2^{0,1} \rightarrow E_2^{2,0}$.  Although the results in Section \ref{examples-section} suggest that $d_2^{0,q}$ vanishes for these components, this problem remains open for Bianchi groups in general.
\end{remark}

\section{Example calculations}\label{examples-section}

In the examples computed in this section,
 we use Corollary~\ref{cor:E2page} and the lemmata of Section~\ref{Calculation of the non-polynomial part of the cohomology}
 to determine the $E_2$~page of the equivariant spectral sequence converging to the cohomology of $\SLtwo(\ringOm)$.
Since the equivariant spectral sequence collapses at the $E_3$ page, once we understand the $d_2^{\rm E3SL}$~differential, 
we can read off $\Cohomol^*(\SLO ; \thinspace \F_2)$. 
On components of type $\edgegraph$, results in Section~\ref{The second page differential on the components of the torsion subcomplex}
show that this $d_2$~differential vanishes. 
This can also be said on components of types $\graphTwo$ and $\graphFive$, except for that we were not able to establish this for~$d_2^{0,4k+1}$
(cf. Remark~\ref{36}).
On components of type $\circlegraph$, Lemma~\ref{lem:CircleExt} shows that we need only to know the rank of~$d_2^{0,1}$, 
for which Appendix~\ref{Numerical results} provides numerical results.  
 When we cannot completely determine the rank of~$d_2$, 
or when the value of the co-rank~$c$ is unclear (see Notation~\ref{notation:c}),
we still can use Bianchi.gp~\cite{BianchiGP} and HAP~\cite{HAP} to compute 
$\Cohomol^*(\SLO ; \thinspace \F_2)$
in order to resolve this indeterminacy.

\subsection*{Example}\textbf{$(\iota)$.}  \label{iota-example}  Let $_\Gamma \backslash X_s = \edgegraph$.
Then $v=2$, $\chi(_\Gamma \backslash X_s) = 1$ and $c = 0$. 
By Lemma~\ref{iota-case}, the differential $d_2^{\rm E3SL}$ vanishes.
Then applying Corollary~\ref{cor:E2page}, Note~\ref{sub-sub} and the proof of Lemma~\ref{lem:iota},
 we obtain the following dimensions for $\Cohomol^*(\Gamma)$.
$$
\dim_{\F_2}\Cohomol^{q}(\SLO; \thinspace \F_2)
=
\begin{cases}
   \beta^{1} (_\Gamma \backslash X) +\beta^{2} (_\Gamma \backslash X), &  q = 4k+5, \\
   \beta^1(_\Gamma \backslash X) +\beta^{2} (_\Gamma \backslash X)+2, &  q = 4k+4, \\
   \beta^{1} (_\Gamma \backslash X) + \beta^{2} (_\Gamma \backslash X)+3, &  q = 4k+3, \\
   \beta^1(_\Gamma \backslash X) +\beta^{2} (_\Gamma \backslash X)+1, &  q = 4k+2, \\
   \beta^{1} (_\Gamma \backslash X) ,  &  q = 1. \\
\end{cases}
$$

The type $_\Gamma \backslash X_s = \edgegraph$ occurs for instance for the cases 
$\Gamma = \SLO$ with 
\\
$m \in \{11, 19, 43, 67, 163\}$,
where $\beta^2=\beta_2 = \beta_1 -1$ and $\beta^1=\beta_1$ is given as follows.
$$\begin{array}{l|ccccc}
m        &  11 & 19 & 43 & 67 & 163 \\
 \hline
\beta_1  & 1 & 1   & 2  &  3 & 7.   \\
\end{array}$$ 
In these cases, the results for $\dim_{\F_2}\Cohomol^{q}(\SLO; \thinspace \F_2)$
 have been checked numerically using Bianchi.gp and HAP.

\subsection*{Example}\textbf{$(\theta)$.} Let $_\Gamma \backslash X_s = \graphFive$.
Then $v=2$ and $\chi(_\Gamma \backslash X_s) = -1$. By Lemma~\ref{theta-case}, the differential $d_2^{0,3}$ vanishes.
Applying Corollary~\ref{cor:E2page}, Note~\ref{sub-sub} and the proof of Lemma~\ref{lem:theta}, 
we obtain the following dimensions for  $\Cohomol^*(\Gamma)$.
$$
\dim_{\F_2}\Cohomol^{q}(\SLO; \thinspace \F_2)
=
\begin{cases}
   c+ \beta^{1} (_\Gamma \backslash X) +\beta^{2} (_\Gamma \backslash X)+2 -\rank d_2^{0,1}, &  q = 4k+5, \\
   c +\beta^1(_\Gamma \backslash X) +\beta^{2} (_\Gamma \backslash X)+2, &  q = 4k+4, \\
   c +\beta^{1} (_\Gamma \backslash X) + \beta^{2} (_\Gamma \backslash X)+3, &  q = 4k+3, \\
   c+\beta^1(_\Gamma \backslash X) +\beta^{2} (_\Gamma \backslash X)+3 -\rank d_2^{0,1}, &  q = 4k+2, \\
   \beta^{1} (_\Gamma \backslash X)+2 -\rank d_2^{0,1},  &  q = 1. \\
\end{cases}
$$
Instances are $\Gamma = \SLO$ with $m = $
$\begin{cases}
   5.                                  & {\rm Then} \medspace \beta^1 = \beta_1 = 2 \medspace {\rm and} \medspace \beta^2 = \beta_2 = 1.\\
   10 \medspace {\rm or} \medspace 13. & {\rm Then} \medspace \beta^1 = \beta_1 = 3 \medspace {\rm and} \medspace \beta^2 = \beta_2 = 2.\\ 
   58.                                 & {\rm Then} \medspace \beta^1 = \beta_1 = 12 \medspace {\rm and} \medspace \beta^2 = \beta_2 = 11.\\                           
                              \end{cases}$
\\ \normalsize
The authors' numerical calculations yield $d_2^{0,1} = 0$,
at an also vanishing co-rank $c$ in all these four cases.

\subsection*{Example}\textbf{$(o)$.}
For the case $_\Gamma \backslash X_s = \circlegraph$, 
we observe that $v$ and $\chi(_\Gamma \backslash X_s)$ are zero.
Corollary~\ref{cor:E2page} and the proof of Lemma~\ref{lem:lemOne}
yield just two cases on the $E_2$-page: 
$E_2^{n,q}(X) \cong \Cohomol^n(_\Gamma \backslash X)+c$ for $n \in \{1, 2\}$ and $q$ odd,
$E_2^{n,q}(X) \cong \Cohomol^n(_\Gamma \backslash X)$ otherwise.

Let $\Gamma = \SLO$ with $m = 7$.
Making use of the fact that in this case the quotient space $_\Gamma \backslash X$ 
is homotopy equivalent to a M\"obius strip \cite{SchwermerVogtmann},
$\beta^1 = \beta_1 = 1$, $\beta^2 = \beta_2 = 0$ and hence $d_2 = 0$.
From the cell structure with stabilizers and identifications, 
we easily see that the co-rank $c$ vanishes.
This allows us to conclude that
$\Cohomol^{q}(\leftSL{-7}\rightSL; \thinspace \F_2)
\cong  (\F_2)^2$ for all $q \geq 1$.

Let $\Gamma = \SLO$ with $m = 15$. 
Then $\beta^1 = \beta_1 = 2$ and $\beta^2 = \beta_2 = 1$.
The numerical computation yields
$\dim_{\F_2}\Cohomol^{q}(\leftSL{-15}\rightSL; \thinspace \F_2)
=$
$
\begin{cases}
  4, &  q \geq 2, \\
  3,  &  q = 1. \\
\end{cases}
$\normalsize \\
We infer that again both $c$ and that the $d_2$--differential vanish.


Let $\Gamma = \SLO$ with $m = $
$\begin{cases}
   35. \medspace {\rm Then} \medspace \beta^1 = \beta_1 = 3 \medspace {\rm and} \medspace \beta^2 = \beta_2 = 2.\\
   91. \medspace {\rm Then} \medspace \beta^1 = \beta_1 = 5 \medspace {\rm and} \medspace \beta^2 = \beta_2 = 4.\\
   115. \medspace {\rm Then} \medspace \beta^1 = \beta_1 = 7 \medspace {\rm and} \medspace \beta^2 = \beta_2 = 6.                            
                              \end{cases}$ \normalsize

The numerical computation yields
$\dim_{\F_2}\Cohomol^{q}(\leftSL{-35}\rightSL; \thinspace \F_2)
= \begin{cases}
  5, &  q \geq 2, \\
  3,  &  q = 1, \\
\end{cases}$ 
\begin{center}
$\dim_{\F_2}\Cohomol^{q}(\leftSL{-91}\rightSL; \thinspace \F_2)
= $ 
$
\begin{cases}
  9, &  q \geq 2, \\
  5,  &  q = 1, \\
\end{cases}
$\normalsize  \qquad
$ {\rm and} \medspace 
\dim_{\F_2}\Cohomol^{q}(\leftSL{-115}\rightSL; \thinspace \F_2)
=$
$
\begin{cases}
  13, &  q \geq 2, \\
  7,  &  q = 1. \\
\end{cases}
$\normalsize
\end{center}
So we infer that $c = 0$ and that the $d_2$--differential has full rank on the column $n=0$
both in rows $q = 4k+1$ and $q = 4k+3$, while it is zero on classes in even rows.

\subsection*{Example}\textbf{$(\rho)$.} 
On $_\Gamma \backslash X_s := \graphTwo$, 
we observe that $v = 2$ and $\chi(_\Gamma \backslash X_s) = 0$.
Applying Corollary~\ref{cor:E2page}, Note~\ref{sub-sub} and the proof of Lemma~\ref{lem:rho}, 
we obtain the following dimensions for the $E_2^{n,q}$--page. 
$$
\begin{array}{l | cccl}
q = 4k+3 & 2     &  \beta^1(_\Gamma \backslash X) +1+c &  \beta^{2} (_\Gamma \backslash X)+c  \\
q = 4k+2 & 2     &  \beta^{1} (_\Gamma \backslash X) +1 &  \beta^{2} (_\Gamma \backslash X) \\
q = 4k+1 & 1     &  \beta^1(_\Gamma \backslash X)+c &  \beta^{2} (_\Gamma \backslash X)+c  \\
q = 4k   & 1     &  \beta^{1} (_\Gamma \backslash X) &\beta^{2} (_\Gamma \backslash X) \\
\hline k \in \mathbb{N} \cup \{0\} &  n = 0 & n = 1 &  n = 2
\end{array}
$$
By Lemma~\ref{theta-case}, the differential $d_2^{0,3}$ vanishes.
This allows us to conclude that
\begin{center}\mbox{ $\dim_{\F_2}\Cohomol^{q}(\leftSL{-m}\rightSL; \thinspace \F_2)
=
\begin{cases}
   \beta^1(_\Gamma \backslash X) +\beta^2(_\Gamma \backslash X) +c+1 -\rank d_2^{0,1}, &  q = 4k+5, \\
   \beta^1(_\Gamma \backslash X) +\beta^2(_\Gamma \backslash X) +c+2, &  q = 4k+4, \\
   \beta^1(_\Gamma \backslash X) +\beta^2(_\Gamma \backslash X) +c+3, &  q = 4k+3, \\
   \beta^1(_\Gamma \backslash X) +\beta^2(_\Gamma \backslash X) +c+2 -\rank d_2^{0,1}, &  q = 4k+2, \\
   \beta^1(_\Gamma \backslash X)  +1 -\rank d_2^{0,1},  &  q = 1. \\
\end{cases}
$\normalsize
}\end{center}

For the first example of case ($\rho$), namely $\Gamma = \SLO$ with  $m = 2$, 
the orbit space $_\Gamma \backslash X$ is homotopy equivalent to a cylinder  \cite{SchwermerVogtmann},
so $\beta^{1} (_\Gamma \backslash X) = 1$, $\beta^{2} (_\Gamma \backslash X) = 0$ and $d_2^{0,1} = 0$.
From the cell structure with stabilizers and identifications, 
we easily see that $c = 0$ in this case.
\\
The examples of Euclidean rings $\calO_{-2}$, $\calO_{-7}$ and $\calO_{-11}$ 
have been checked in HAP~\cite{HAP}
with the cellular complex imported from Bianchi.gp~\cite{BianchiGP}; 
and the example $\calO_{-2}$ additionally by a paper-and-pencil calculation by the first author using classical methods.
We also observe that these three Euclidean examples are compatible with the homology with Steinberg coefficients calculated
in~\cite{SchwermerVogtmann} 
up to a minor typo present in the latter paper in the case $\calO_{-2}$.  

\vfill

\subsection*{Example}\textbf{$(o$ $o$ $o)$.}  Let $_\Gamma \backslash X_s = \circlegraph \circlegraph \circlegraph$.
Then $v = 0$ and $\chi(_\Gamma \backslash X_s) = 0$, yielding the $E_2$ page dimensions
\\  
$$  
\begin{array}{l | cccl}
q = 4k+3 & 3     &   \beta^1(_\Gamma \backslash X) +2 +c  &  \beta^{2} (_\Gamma \backslash X) +c \\
q = 4k+2 & 1     &   \beta^{1} (_\Gamma \backslash X) &  \beta^{2} (_\Gamma \backslash X)  \\
q = 4k+1 & 3     &  \beta^1(_\Gamma \backslash X)+2 +c &  \beta^{2} (_\Gamma \backslash X) +c \\
q = 4k   & 1     &  \beta^{1} (_\Gamma \backslash X) & \beta^{2} (_\Gamma \backslash X) \\
\hline k \in \mathbb{N} \cup \{0\} &  n = 0 & n = 1 &  n = 2
\end{array}
$$

Lemma~\ref{Lem:EvenDegreeTrivial} yields rank$d_2^{0,2} = 0$.
\medskip

This case is realized for $\Gamma = \SLO$ with 
\begin{itemize}
\item $m = 235$. Then $\beta^1=\beta_1=13$ and $\beta^2=\beta_2=12$.
\\
With the HAP implementation, we obtain
\begin{center}$\dim_{\F_2}\Cohomol^{q}(\leftSL{-235}\rightSL; \thinspace \F_2)
=$
$
\begin{cases}
   27, &  q \geq 2, \\
   14,  &  q = 1. \\
\end{cases}$ \normalsize
\end{center}

\item $m = 427$. Then $\beta_1 = \beta^1 = 21$, 
$\beta_2 = \beta^2 = 20$ and the machine obtains 
\begin{center}
 $\dim_{\F_2}\Cohomol^{q}(\leftSL{-427}\rightSL; \thinspace \F_2)
=$ 
$
\begin{cases}
   43, &  q \geq 2, \\
   22,  &  q = 1. \\
\end{cases}$
\end{center}
\end{itemize}
\normalsize
In both cases, we infer that $c = 1$ and we confirm rank~$d_2^{0,1} = $ rank~$d_2^{0,3} = 2$ 
that we obtain from Appendix~\ref{Numerical results} and Lemma~\ref{lem:CircleExt}.

\newpage

\subsection*{Example}\textbf{$(\theta$ $o$ $o)$.}  Let $_\Gamma \backslash X_s = \graphFive \circlegraph \circlegraph$.
Then $v = 2$ and $\chi(_\Gamma \backslash X_s) = -1$, yielding on the $E_2$ page the dimensions
\\
\begin{wrapfigure}[27]{r}{54mm}
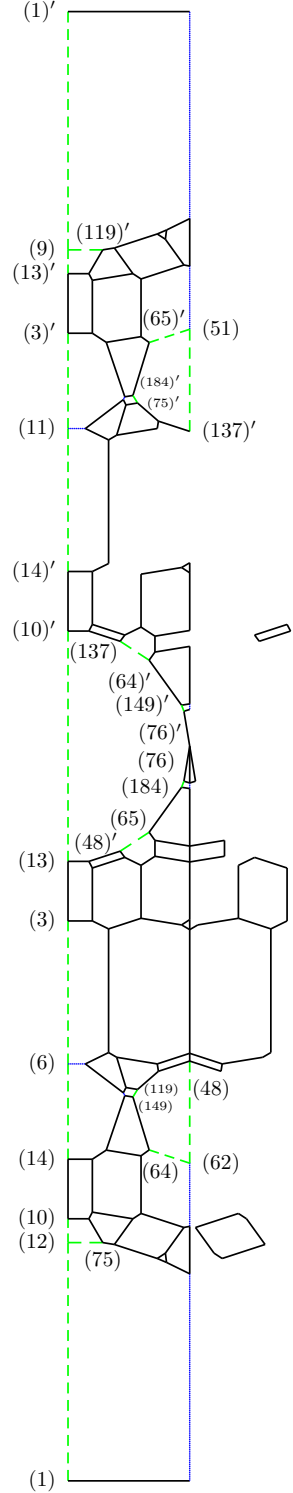

 \centering
 \vspace{-10pt}
   \fundamentalDomain
   \caption{Fundamental domain in the case \mbox{$m = 37$}.}
\label{plot}
 \vspace{-10pt}
\end{wrapfigure} 
$  
\begin{array}{l | cccl}
q = 4k+3 & 4     &  \beta^1(_\Gamma \backslash X) +c+3  &  \beta^{2} (_\Gamma \backslash X) +c \\
q = 4k+2 & 4     &   \beta^{1} (_\Gamma \backslash X) +1 &  \beta^{2} (_\Gamma \backslash X)  \\
q = 4k+1 & 4    &  \beta^1(_\Gamma \backslash X) +c+1  &  \beta^{2} (_\Gamma \backslash X) +c \\
q = 4k   & 1  &  \beta^{1} (_\Gamma \backslash X) & \beta^{2} (_\Gamma \backslash X) \\
\hline k \in \mathbb{N} \cup \{0\} &  n = 0 & n = 1 &  n = 2.
\end{array}
$ \vfill

This case is realized for $\Gamma = \SLO$ with $m = 37$. Then $\beta^1=\beta_1=8$ and $\beta^2=\beta_2=7$.
With the HAP implementation, we obtain
\\ 
$\dim_{\F_2}\Cohomol^{q}(\leftSL{-37}\rightSL; \thinspace \F_2)
=$ \scriptsize$
\begin{cases}
   19, &  q = 4k+5, \\
   19, &  q = 4k+4, \\
   20, &  q = 4k+3, \\
   20, &  q = 4k+2, \\
   11,  &  q = 1, \\
\end{cases}
$\normalsize
\\ 
so we conclude that rank $d_2^{0,1} = $ rank $d_2^{0,3} = 1$ and that $c = 1$.
To illustrate the difficulty in manually determining the co-rank $c$, 
we print, in Figure~\ref{plot}, a fundamental domain for the non-central $2$--torsion subcomplex 
as a dashed graph
contained in the boundary of the Bianchi fundamental polyhedron for the action of $\leftSL{-37}\rightSL$
on $X$. In this figure, vertices labeled by Bianchi.gp with the same number are identified.
However, in order to preserve readability, only the vertices of the non-central $2$--torsion subcomplex are labeled;
and identifications of cells outside of it need to be taken into account for the Betti number $\beta_1=8$.

\subsection*{Example}\textbf{$(o$ $\iota)$.} Let $_\Gamma \backslash X_s = \circlegraph \edgegraph$.
Then $v = 2$ and $\chi(_\Gamma \backslash X_s) = 1$. 
Using our calculations in the above examples, we obtain the following dimensions for the $E_2^{n,q}$--page.
 \\  \vfill
$
\begin{array}{l | cccl}
q = 4k+3 & 3     &  \beta^1(_\Gamma \backslash X) +c+2 &  \beta^{2} (_\Gamma \backslash X)+c  \\
q = 4k+2 & 0     &  \beta^{1} (_\Gamma \backslash X) +1 &  \beta^{2} (_\Gamma \backslash X)  \\
q = 4k+1 & 1     &  \beta^1(_\Gamma \backslash X) +c+2 &  \beta^{2} (_\Gamma \backslash X)+c  \\
q = 4k   & 1     &  \beta^{1} (_\Gamma \backslash X)c& \beta^{2} (_\Gamma \backslash X) \\
\hline k \in \mathbb{N} \cup \{0\} &  n = 0 & n = 1 &  n = 2
\end{array}
$
 \\ \vfill
This case is realized for $\Gamma = \SLO$ with \\
$m = \begin{cases}
6.  &  {\rm Then} \medspace \beta^1=\beta_1=2 \medspace {\rm and} \medspace \beta^2=\beta_2=1.\\
22. &  {\rm Then} \medspace \beta^1=\beta_1=5 \medspace {\rm and} \medspace \beta^2=\beta_2=4.\\          
                                                               \end{cases}$
 \\
With the HAP implementation, we obtain
\\ 
$\dim_{\F_2}\Cohomol^{q}(\leftSL{-6}\rightSL; \thinspace \F_2)
=$ \scriptsize$
\begin{cases}
   4, &  q = 4k+5, \\
   6, &  q = 4k+4, \\
   7, &  q = 4k+3, \\
   5, &  q = 4k+2, \\
   3,  &  q = 1, \\
\end{cases}
$\normalsize
\\
$\dim_{\F_2}\Cohomol^{q}(\leftSL{-22}\rightSL; \thinspace \F_2)
=$ \scriptsize$
\begin{cases}
   10, &  q = 4k+5, \\
   12, &  q = 4k+4, \\
   13, &  q = 4k+3, \\
   11, &  q = 4k+2, \\
   6,  &  q = 1, \\
\end{cases}
$\normalsize
\\ 
so we conclude that $c = 0$ and $d_2 = 0$.

\newpage

\section{Comparison with the second page differential of PSL} \label{The second page differential}
As the $d_2^{\rm E3SL}$ differential is not yet completely determined 
(we do not know a priori which rank it has in odd degrees on components of type $\circlegraph$),
this section compares its rank with the dimension
over $\F_2$ of the $2$--primary part of the image of the $d^2_{2,0}$--differential 
 of the homological equivariant spectral sequence associated to $\PSLtwo(\ringOm)$ 
 on  the cell complex $X$ with integral coefficients.
 We denote the latter dimension by $D$, and set 
 $\Gamma := \mathrm{\PSLtwo}(\calO_{-m})$ for this section.
We compute the dimension $D$ numerically, using the following arguments.
We use the cell structure which is subdivided until each cell $\sigma$ is fixed point-wise by its stabilizer $\Gamma_\sigma$.
The $E^2_{p,q}$--page of this spectral sequence is concentrated as follows in the three columns $p\in \{0,1,2 \}$,
\scriptsize
$$ \xymatrix{
q > 1 &  2\mbox{-}{\rm torsion} \oplus 3\mbox{-}{\rm torsion} &  2\mbox{-}{\rm torsion} \oplus 3\mbox{-}{\rm torsion} & 0\\
q=1 & {\rm Farrell \medspace supplement} & 2\mbox{-}{\rm torsion} \oplus 3\mbox{-}{\rm torsion} & 0 \\
q=0 &{\mathbb{Z}} & \Homol_1(_\Gamma\backslash{X}; \thinspace \Z) & \ar[ull]_{d^2_{2,0}} \Homol_2(_\Gamma\backslash{X}; \thinspace \Z) 
}$$
\normalsize
where the ``Farrell supplement'' is the cokernel of the map
$$
\bigoplus_{\sigma \thinspace\in \thinspace _\Gamma\backslash X^{(0)}} \Homol_1 (\Gamma_\sigma; \Z)
\xleftarrow{\ d^1_{1,1}\ }
\bigoplus_{\sigma \thinspace\in \thinspace _\Gamma\backslash X^{(1)}} \Homol_1(\Gamma_\sigma; \Z)
$$
induced by inclusion of cell stabilizers.
The Farrell supplement and $\Homol_1(_\Gamma\backslash{X}; \thinspace \Z)$ have been computed on a database of Bianchi groups~\cite{higher_torsion}.

For the cases in this database for which the origin of $d^2_{2,0}$ is nontrivial and its target contains $2$--torsion,
Aurel Page has computed the abelianization $\Gamma^{ab} \cong \Homol_1(\Gamma; \thinspace \Z)$, 
i.e. the commutator factor subgroup, of  $\Gamma = \mathrm{\PSLtwo}(\calO_{-m})$.
As the above spectral sequence converges to the group homology of $\Gamma$ with integer coefficients,
we obtain a short exact sequence
$$ 0 \to  {\rm Farrell \medspace supplement}/_{\image d^2_{2,0}} \to \mathrm{\PSLtwo}(\calO_{-m})^{ab} \to \Homol_1(_\Gamma\backslash{X}; \thinspace \Z) \to 1;$$
and we deduce from it the image of $d^2_{2,0}$ in Appendix ~\ref{Numerical results}.

\begin{observation}
\begin{itemize}
 \item Within the scope of the database in Appendix ~\ref{Numerical results}, 
the dimension $D$
 is at most the number of connected components of type~$\circlegraph$ in the non-central $2$--torsion subcomplex quotient. 
 It is clear from $\Homol_1(\Af ; \thinspace \Z) \cong \Z/3$ 
 that the target of $d^2_{2,0}$ has no $2$--torsion on the connected components of type $\edgegraph$,
 and the same property follows for connected components of type $\graphTwo$ 
 by a lemma of~\cite{RahmTorsion}\thinspace{}
 specifying the matrix block induced in the $d^1_{1,1}$--differential by the inclusions into $\Dtwo$.
 The connected components on which our observation is backed only by the numerical results
 are the ones of type $\graphFive$.
 \item In the scope of the table in the appendix, 
 the dimension $D$ 
 agrees with the rank of the $d_2^{0,1}$--differential 
 of the cohomological equivariant spectral sequence with $\F_2$--coefficients associated to 
 $\SLtwo(\ringOm)$.
\end{itemize}
\end{observation}

\vfill

\appendix

\section{Numerical results} \label{Numerical results}
The program~\cite{AurelPage} computes a presentation of the Bianchi groups.
Aurel Page has carried this out and calculated the commutator factor groups (abelianizations) $\SLO^{\rm Ab}$ and $\PSLO^{\rm Ab}$
in the cases that we need in order to deduce the image of~$d^2_{2,0}$ in the way 
described in Section~\ref{The second page differential}\thinspace{}.
The outcome of this procedure is included in the below tables.
We denote by $\Delta$ the discriminant of $\calO_{-m}$, i.e.
$$\Delta = \begin{cases} -m, & m \equiv 3 \bmod 4, \\
-4m,  & \mathrm{otherwise.} \end{cases}$$
 and by $\beta_1$ the first Betti number of both $\PSLO$ and $\SLO$.
This Betti number has been computed with the two independent programs~\cite{AurelPage} and~\cite{BianchiGP} with identical results. 
We provide the remaining torsion parts $\SLO^{\rm Ab}_{\rm tors}$ and $\PSLO^{\rm Ab}_{\rm tors}$ obtained with~\cite{AurelPage}.
We insert the quotient of a reduced non-central $2$--torsion subcomplex from~\cite{AccessingFarrell} and~\cite{BianchiGP}:
\\
Let $o$ denote the number of connected components of homeomorphism type $\circlegraph$, 
let $\iota$ the one for $\edgegraph$,
$\theta$ for $\graphFive$ and $\rho$ for $\graphTwo$.
In the cases where the d\'evissage 
(extension problem between the last page of the equivariant spectral sequence and group homology)
of the $2$-torsion part of $\PSLO^{\rm Ab}$ is trivial,
using Corollary~\ref{cor:E2page} and Lemmata~\ref{lem:lemOne} through~\ref{lem:rho},
we deduce the rank of $d_2^{0,1}$ for the equivariant spectral sequence with $\F_2$--coefficients associated to the action of
$\SLO$ on $X$ as the difference 
$$ \rank_{\F_2} d_2^{0,1} = o +2\theta +\rho +\dim_{\F_2} 
\left(\text{Hom}(\operatorname{H}_1(_\Gamma\backslash X; \Z)_{\rm tors}, \F_2)\right) 
-\dim_{\F_2} \left(\text{Hom}(\SLO^{\rm Ab}_{\rm tors}, \F_2)\right).$$
We can see that the d\'evissage of the $2$-torsion part of $\PSLO^{\rm Ab}$ is trivial when
\begin{itemize}
 \item each summand $\Z/(2^r)$ of $\PSLO^{\rm Ab}$ has $r = 1$,
 \item or when $\operatorname{H}_1(_\Gamma\backslash X; \Z)$ admits no $2$-torsion.
\end{itemize}
These two criteria allow us to apply the above formula to all cases in our database except for
$m \in \{142, 1227, 1411, 1555\}$.
In those four cases, the situation is somewhat more complicated.

By Lemma~\ref{lem:CircleExt},
we know furthermore that the rank of $d_2^{0,1}$ is the rank of  $d_2^{0,q}$ in all odd degrees $q = 2k+1$
on components of type $\circlegraph$.
\newpage

\scriptsize
$$
\begin{array}{|l|l|c|c|c|c|c|c|c|c|c|}
\hline
 -\Delta       & m         &    o       &    \iota      &      \theta   &     \rho      &     \beta_1   &   \PSLO^{\rm Ab}_{\rm tors}                           & \image d^2_{2,0} ({\rm PSL}) & \SLO^{\rm Ab}_{\rm tors} &  \rank_{\F_2} d_2^{0,1} ({\rm SL})
\\ \hline  & & & & & & & & & &
\\ 35 	&	35	&	1	&	0	&	0	&	0	&	3	&	  \Z/3      																				&	 \Z/2                        	&	  \Z/3                               			&	1
\\ 40 	&	10	&	0	&	0	&	1	&	0	&	3	&	  (\Z/2)^2      																				&	 \Z/3                        	&	  (\Z/2)^2                           			&	0
\\ 47 	&	47	&	1	&	0	&	0	&	0	&	5	&	  \Z/2  \oplus					  \Z/3      															&	0	&	  \Z/4  \oplus	  \Z/3   		&	0
\\ 52 	&	13	&	0	&	0	&	1	&	0	&	3	&	  (\Z/2)^2      																				&	0	&	  (\Z/2)^2                           			&	0
\\ 55 	&	55	&	1	&	0	&	0	&	0	&	5	&	  \Z/2      																				&	 \Z/3                        	&	  \Z/4                               			&	0
\\ 56 	&	14	&	2	&	0	&	0	&	0	&	5	&	  \Z/2  \oplus					  \Z/3      															&	 \Z/2                        	&	  \Z/2  \oplus	  \Z/3   		&	1
\\ 68 	&	17	&	0	&	1	&	1	&	0	&	5	&	  (\Z/2)^2        \oplus					  \Z/3 															&	0	&	  (\Z/2)^2  \oplus	  \Z/3    		&	0
\\ 79 	&	79	&	3	&	0	&	0	&	0	&	6	&	  (\Z/2)^2      																				&	 \Z/2                        	&	  \Z/2  \oplus	  \Z/4               		&	1
\\ 84 	&	21	&	3	&	0	&	0	&	0	&	7	&	  (\Z/2)^2        \oplus					  \Z/3 															&	 \Z/2 \oplus \Z/3            	&	  (\Z/2)^2  \oplus	  \Z/3    		&	1
\\ 87 	&	87	&	1	&	0	&	0	&	0	&	8	&	  \Z/2  \oplus					  \Z/3      															&	0	&	  \Z/4  \oplus	  \Z/3   		&	0
\\ 88 	&	22	&	1	&	1	&	0	&	0	&	5	&	  \Z/2      																				&	 \Z/3                        	&	  \Z/2                               			&	0
\\ 91 	&	91	&	1	&	0	&	0	&	0	&	5	&	0	&	 \Z/2                        	&	0	&	1
\\ 95 	&	95	&	1	&	0	&	0	&	0	&	9	&	  \Z/2  \oplus					  \Z/3      															&	0	&	  \Z/4  \oplus	  \Z/3   		&	0
\\ 103 	&	103	&	1	&	0	&	0	&	0	&	7	&	  \Z/2      																				&	0	&	  \Z/4                             			&	0
\\ 104 	&	26	&	0	&	0	&	1	&	0	&	8	&	  (\Z/2)^2        \oplus					  \Z/3 															&	 \Z/3                        	&	  (\Z/2)^2  \oplus	  \Z/3   		&	0
\\ 111 	&	111	&	1	&	0	&	0	&	0	&	10	&	  \Z/2  \oplus					  \Z/3 															&	0	&	  \Z/4  \oplus	  \Z/3    		&	0
\\ 115 	&	115	&	1	&	0	&	0	&	0	&	7	&	0	&	 \Z/2 \oplus \Z/3            	&	0	&	1
\\ 116 	&	29	&	0	&	0	&	1	&	0	&	9	&	  (\Z/2)^2  \oplus					  \Z/3 															&	0	&	  (\Z/2)^2  \oplus	  \Z/3   		&	0
\\ 119 	&	119	&	2	&	0	&	0	&	0	&	11	&	  \Z/2  \oplus					  \Z/3 															&	 \Z/2                        	&	  \Z/4  \oplus	  \Z/3    		&	1
\\ 120 	&	30	&	3	&	0	&	0	&	0	&	10	&	  \Z/2  \oplus					  \Z/3 															&	 (\Z/2)^2 \oplus (\Z/3)^2    	&	  \Z/2  \oplus	  \Z/3     		&	2
\\ 127 	&	127	&	1	&	0	&	0	&	0	&	8	&	  \Z/2      																				&	0	&	  \Z/4                             			&	0
\\ 132 	&	33	&	3	&	2	&	0	&	0	&	10	&	  (\Z/2)^2  \oplus					  \Z/3 															&	 \Z/2 \oplus (\Z/3)^2        	&	  (\Z/2)^2  \oplus	  \Z/3 		&	1
\\ 136 	&	34	&	2	&	0	&	0	&	2	&	8	&	  (\Z/2)^2      																				&	 (\Z/2)^2 \oplus \Z/3        	&	  (\Z/2)^2                          			&	2
\\ 143 	&	143	&	1	&	0	&	0	&	0	&	12	&	  \Z/2  \oplus					  \Z/3 															&	 \Z/3                        	&	  \Z/4  \oplus	  \Z/3    		&	0
\\ 148 	&	37	&	2	&	0	&	1	&	0	&	8	&	  (\Z/2)^3      																				&	 \Z/2                        	&	  (\Z/2)^3                          			&	1
\\ 151 	&	151	&	1	&	0	&	0	&	0	&	10	&	  \Z/2      																				&	0	&	  \Z/4                           			&	0
\\ 152 	&	38	&	1	&	1	&	0	&	0	&	10	&	  \Z/2  \oplus					  \Z/3 															&	0	&	  \Z/2  \oplus	  \Z/3     		&	0
\\ 155 	&	155	&	1	&	0	&	0	&	0	&	10	&	  \Z/3      																				&	 \Z/2                        	&	  \Z/3                           			&	1
\\ 159 	&	159	&	1	&	0	&	0	&	0	&	14	&	  \Z/2  \oplus					  \Z/3 															&	0	&	  \Z/4  \oplus	  \Z/3    		&	0
\\ 164 	&	41	&	0	&	1	&	1	&	0	&	12	&	  (\Z/2)^2  \oplus					  \Z/3 															&	0	&	  (\Z/2)^2  \oplus	  \Z/3 		&	0
\\ 167 	&	167	&	1	&	0	&	0	&	0	&	13	&	  \Z/2  \oplus					  \Z/3 															&	0	&	  \Z/4  \oplus	  \Z/3    		&	0
\\ 168 	&	42	&	3	&	0	&	0	&	0	&	13	&	  \Z/2  \oplus					  \Z/3 															&	 (\Z/2)^2 \oplus \Z/3        	&	  \Z/2  \oplus	  \Z/3     		&	2
\\ 183 	&	183	&	1	&	0	&	0	&	0	&	14	&	  \Z/2  \oplus					  \Z/3 															&	0	&	  \Z/4  \oplus	  \Z/3    		&	0
\\ 184 	&	46	&	2	&	0	&	0	&	0	&	11	&	  \Z/2      																				&	 \Z/2 \oplus \Z/3            	&	  \Z/2                            			&	1
\\ 191 	&	191	&	1	&	0	&	0	&	0	&	15	&	  \Z/2  \oplus					  \Z/3 															&	0	&	  \Z/4  \oplus	  \Z/3    		&	0
\\ 195 	&	195	&	2	&	0	&	0	&	0	&	15	&	  \Z/3      																				&	 (\Z/2)^2 \oplus \Z/3        	&	  \Z/3                           			&	2
\\ 199 	&	199	&	1	&	0	&	0	&	0	&	13	&	  \Z/2      																				&	0	&	  \Z/4                           			&	0
\\ 203 	&	203	&	1	&	0	&	0	&	0	&	12	&	  \Z/3      																				&	 \Z/2                        	&	  \Z/3                           			&	1
\\ 212 	&	53	&	0	&	0	&	1	&	0	&	14	&	  (\Z/2)^2  \oplus					  \Z/3 															&	0	&	  (\Z/2)^2  \oplus	  \Z/3 		&	0
\\ 215 	&	215	&	1	&	0	&	0	&	0	&	18	&	  \Z/2  \oplus					  \Z/3 															&	0	&	  \Z/4  \oplus	  \Z/3    		&	0
\\ 219 	&	219	&	1	&	2	&	0	&	0	&	13	&	  \Z/3      																				&	 \Z/2                        	&	  \Z/3                           			&	1
\\ 223 	&	223	&	3	&	0	&	0	&	0	&	15	&	  \Z/2      																				&	 (\Z/2)^2                    	&	  \Z/4                           			&	2
\\ 228 	&	57	&	3	&	2	&	0	&	0	&	16	&	  (\Z/2)^2  \oplus					  \Z/3 															&	 \Z/2 \oplus \Z/3        	&	  (\Z/2)^2  \oplus	  \Z/3 		&	1
\\ 231 	&	231	&	2	&	0	&	0	&	0	&	21	&	  \Z/2  \oplus					  \Z/3 															&	 \Z/2 \oplus \Z/3            	&	  \Z/4  \oplus	  \Z/3    		&	1
\\ 232 	&	58	&	0	&	0	&	1	&	0	&	12	&	  (\Z/2)^2      																				&	 \Z/3                        	&	  (\Z/2)^2                        			&	0
\\ 235 	&	235	&	3	&	0	&	0	&	0	&	13	&	  \Z/2      																				&	 (\Z/2)^2 \oplus \Z/3        	&	  \Z/2                           			&	2
\\ 239 	&	239	&	1	&	0	&	0	&	0	&	18	&	  \Z/2  \oplus					  \Z/3 															&	0	&	  \Z/4  \oplus	  \Z/3    		&	0
\\ 244 	&	61	&	0	&	0	&	1	&	0	&	15	&	  (\Z/2)^2      																				&	0	&	  (\Z/2)^2                        			&	0
\\ 247 	&	247	&	1	&	0	&	0	&	0	&	14	&	  \Z/2      																				&	0	&	  \Z/4                           			&	0
\\ 248 	&	62	&	2	&	0	&	0	&	0	&	16	&	  \Z/2  \oplus					  \Z/3 															&	 \Z/2                        	&	  \Z/2  \oplus	  \Z/3     		&	1
\\ 255 	&	255	&	2	&	0	&	0	&	0	&	23	&	  \Z/2  \oplus					  \Z/3 															&	 \Z/2 \oplus (\Z/3)^2        	&	  \Z/4  \oplus	  \Z/3    		&	1
\\ 259 	&	259	&	1	&	0	&	0	&	0	&	14	&	0	&	 \Z/2 \oplus \Z/3            	&	0	&	1
\\ 260 	&	65	&	1	&	0	&	2	&	0	&	20	&	  (\Z/2)^4  \oplus					  \Z/3 															&	 \Z/2 \oplus \Z/3            	&	  (\Z/2)^4  \oplus	  \Z/3 		&	1
\\ 263 	&	263	&	1	&	0	&	0	&	0	&	18	&	  \Z/2  \oplus					  \Z/3 															&	 0            	&	  \Z/4  \oplus	  \Z/3    		&	0
\\ 264 	&	66	&	2	&	2	&	0	&	0	&	20	&	  \Z/2  \oplus					  \Z/3 															&	 \Z/2 \oplus (\Z/3)^2        	&	  \Z/2  \oplus	  \Z/3     		&	1
\\ 271 	&	271	&	1	&	0	&	0	&	0	&	17	&	  \Z/2      																				&	0	&	  \Z/4                           			&	0
\\ 276 	&	69	&	3	&	0	&	0	&	0	&	23	&	  (\Z/2)^2  \oplus					  \Z/3 															&	 \Z/2 \oplus \Z/3            	&	  (\Z/2)^2  \oplus	  \Z/3 		&	1
\\ 280 	&	70	&	3	&	0	&	0	&	0	&	19	&	  \Z/2      																				&	 (\Z/2)^2 \oplus (\Z/3)^2    	&	  \Z/2                            			&	2
\\ 287 	&	287	&	2	&	0	&	0	&	0	&	21	&	  \Z/2  \oplus					  \Z/3 															&	 \Z/2                        	&	  \Z/4  \oplus	  \Z/3    		&	1
\\ 291 	&	291	&	1	&	2	&	0	&	0	&	17	&	  \Z/2  \oplus					  \Z/3 															&	0	&	  \Z/2  \oplus	  \Z/3    		&	0
\\ 292 	&	73	&	0	&	1	&	1	&	0	&	16	&	  (\Z/2)^2      																				&	 \Z/3                        	&	  (\Z/2)^2                        			&	0
\\ 295 	&	295	&	1	&	0	&	0	&	0	&	19	&	  \Z/2      																				&	 \Z/3                        	&	  \Z/4                           			&	0
\\ 296 	&	74	&	0	&	0	&	1	&	0	&	19	&	  (\Z/2)^4  \oplus					  \Z/3 															&	 \Z/3                        	&	  (\Z/2)^4  \oplus	  \Z/3 		&	0
\\ 299 	&	299	&	1	&	0	&	0	&	0	&	18	&	  (\Z/3)^2      																				&	 \Z/2 \oplus (\Z/3)^2        	&	  (\Z/3)^2                       			&	1
\\ 303 	&	303	&	1	&	0	&	0	&	0	&	22	&	  \Z/2  \oplus					  \Z/3 															&	0	&	  \Z/4  \oplus	  \Z/3    		&	0
\\ 308 	&	77	&	3	&	0	&	0	&	0	&	23	&	  (\Z/2)^2  \oplus					  \Z/3 															&	 \Z/2 \oplus \Z/3            	&	  (\Z/2)^2  \oplus	  \Z/3 		&	1
\\ 311 	&	311	&	1	&	0	&	0	&	0	&	23	&	  \Z/2  \oplus					  \Z/3 															&	0	&	  \Z/4  \oplus	  \Z/3    		&	0
\\ \hline
\end{array}
$$

$$
\begin{array}{|l|l|c|c|c|c|c|c|c|c|c|}
\hline
 -\Delta & m &      o       &    \iota      &      \theta   &     \rho      &     \beta_1   &   \PSLO^{\rm Ab}_{\rm tors}                           & \image d^2_{2,0} ({\rm PSL}) & \SLO^{\rm Ab}_{\rm tors} &  {\rm rk} d_2^{0,1}
\\ \hline  & & & & & & & & & &
\\ 312 	&	78	&	3	&	0	&	0	&	0	&	22	&	  (\Z/2)^2  \oplus					  \Z/3 															&	 \Z/2 \oplus \Z/3         	&	  (\Z/2)^2  \oplus	  \Z/3                                 		&	1
\\ 319 	&	319	&	1	&	0	&	0	&	0	&	20	&	  \Z/2      																				&	 \Z/3                     	&	  \Z/4                                                           			&	0
\\ 323 	&	323	&	1	&	2	&	0	&	0	&	16	&	  \Z/3      																				&	 \Z/2                     	&	  \Z/3                                                           			&	1
\\ 327 	&	327	&	1	&	0	&	0	&	0	&	24	&	  \Z/2  \oplus					  \Z/3 															&	0	&	  \Z/4  \oplus	  \Z/3                                    		&	0
\\ 328 	&	82	&	1	&	1	&	1	&	0	&	17	&	  (\Z/2)^3      																				&	 \Z/3                     	&	  (\Z/2)^3                                                        			&	0
\\ 335 	&	335	&	1	&	0	&	0	&	0	&	26	&	  \Z/2  \oplus					  \Z/3 															&	0	&	  \Z/4  \oplus	  \Z/3                                    		&	0
\\ 340 	&	85	&	0	&	0	&	2	&	0	&	23	&	  (\Z/2)^4      																				&	 (\Z/3)^2                 	&	  (\Z/2)^4                                                        			&	0
\\ 344 	&	86	&	1	&	1	&	0	&	0	&	21	&	  (\Z/2)^3  \oplus					  \Z/3 															&	0	&	  (\Z/2)^3  \oplus	  \Z/3                                 		&	0
\\ 355 	&	355	&	1	&	0	&	0	&	0	&	20	&	0	&	 \Z/2 \oplus \Z/3         	&	0	&	1
\\ 356 	&	89	&	0	&	1	&	1	&	0	&	24	&	  (\Z/2)^2  \oplus					  \Z/3 															&	0	&	  (\Z/2)^2  \oplus	  \Z/3                                 		&	0
\\ 359 	&	359	&	3	&	0	&	0	&	0	&	25	&	  (\Z/2)^4  \oplus					  \Z/3 															&	 \Z/2                     	&	  (\Z/2)^3  \oplus	  \Z/4  \oplus	  \Z/3                  	&	1
\\ 367 	&	367	&	1	&	0	&	0	&	0	&	20	&	  \Z/2  \oplus					  (\Z/3)^2 															&	 \Z/3                     	&	  \Z/4  \oplus	  (\Z/3)^2                                		&	0
\\ 371 	&	371	&	1	&	0	&	0	&	0	&	22	&	  \Z/3      																				&	 \Z/2                     	&	  \Z/3                                                           			&	1
\\ 372 	&	93	&	3	&	0	&	0	&	0	&	27	&	  (\Z/2)^2  \oplus					  \Z/3 															&	 \Z/2 \oplus \Z/3         	&	  (\Z/2)^2  \oplus	  \Z/3                                 		&	1
\\ 376 	&	94	&	2	&	0	&	0	&	0	&	22	&	  \Z/2      																				&	 \Z/2 \oplus \Z/3         	&	  \Z/2                                                            			&	1
\\ 383 	&	383	&	1	&	0	&	0	&	0	&	25	&	  \Z/2  \oplus					  \Z/3 															&	0	&	  \Z/4  \oplus	  \Z/3                                    		&	0
\\ 388 	&	97	&	0	&	1	&	1	&	0	&	21	&	  (\Z/2)^2      																				&	 \Z/3                     	&	  (\Z/2)^2                                                        			&	0
\\ 391 	&	391	&	2	&	0	&	0	&	0	&	25	&	  \Z/2      																				&	 \Z/2 \oplus \Z/3         	&	  \Z/4                                                           			&	1
\\ 395 	&	395	&	1	&	0	&	0	&	0	&	24	&	  (\Z/2)^2  \oplus					  \Z/3 															&	 \Z/2                     	&	  (\Z/2)^2  \oplus	  \Z/3                                		&	1
\\ 399 	&	399	&	4	&	0	&	0	&	0	&	33	&	  (\Z/2)^2  \oplus					  \Z/3 															&	 (\Z/2)^2 \oplus \Z/3     	&	  \Z/2  \oplus	  \Z/4  \oplus	  \Z/3                    	&	2
\\ 403 	&	403	&	1	&	0	&	0	&	0	&	19	&	0	&	 \Z/2                     	&	0	&	1
\\ 404 	&	101	&	2	&	0	&	1	&	0	&	28	&	  (\Z/2)^3  \oplus					  \Z/3 															&	 \Z/2                     	&	  (\Z/2)^3  \oplus	  \Z/3                                		&	1
\\ 407 	&	407	&	1	&	0	&	0	&	0	&	29	&	  \Z/2  \oplus					  \Z/3 															&	 \Z/3                     	&	  \Z/4  \oplus	  \Z/3                                    		&	0
\\ 408 	&	102	&	2	&	2	&	0	&	0	&	27	&	  \Z/2  \oplus					  (\Z/3)^2 															&	 \Z/2 \oplus (\Z/3)^4     	&	  \Z/2  \oplus	  (\Z/3)^2                                		&	1
\\ 415 	&	415	&	1	&	0	&	0	&	0	&	28	&	  \Z/2      																				&	 \Z/3                     	&	  \Z/4                                                           			&	0
\\ 420 	&	105	&	8	&	0	&	0	&	0	&	41	&	  (\Z/2)^2  \oplus					  \Z/3 															&	 (\Z/2)^6 \oplus (\Z/3)^3 	&	  (\Z/2)^2  \oplus	  \Z/3                                		&	6
\\ 424 	&	106	&	0	&	0	&	1	&	0	&	23	&	  (\Z/2)^4      																				&	 \Z/3                     	&	  (\Z/2)^4                                                       			&	0
\\ 427 	&	427	&	3	&	0	&	0	&	0	&	21	&	  \Z/2      																				&	 (\Z/2)^2                 	&	  \Z/2                                                           			&	2
\\ 431 	&	431	&	1	&	0	&	0	&	0	&	29	&	  (\Z/2)^2  \oplus					  \Z/3 															&	0	&	  \Z/2  \oplus	  \Z/4  \oplus	  \Z/3                    	&	0
\\ 435 	&	435	&	2	&	0	&	0	&	0	&	31	&	  \Z/3      																				&	 (\Z/2)^2 \oplus (\Z/3)^5 	&	  \Z/3                                                           			&	2
\\ 436 	&	109	&	0	&	0	&	1	&	0	&	25	&	  (\Z/2)^4      																				&	0	&	  (\Z/2)^4                                                       			&	0
\\ 439 	&	439	&	5	&	0	&	0	&	0	&	26	&	  (\Z/2)^3      																				&	 (\Z/2)^2                 	&	  (\Z/2)^2  \oplus	  \Z/4                                         		&	2
\\ 440 	&	110	&	3	&	0	&	0	&	0	&	32	&	  (\Z/2)^2  \oplus					  \Z/3 															&	 \Z/2 \oplus \Z/3         	&	  (\Z/2)^2  \oplus	  \Z/3                                		&	1
\\ 443 	&	443	&	1	&	1	&	0	&	0	&	21	&	  \Z/2  \oplus					  \Z/3 															&	0	&	  \Z/2  \oplus	  \Z/3                                    		&	0
\\ 447 	&	447	&	1	&	0	&	0	&	0	&	32	&	  \Z/2  \oplus					  \Z/3 															&	0	&	  \Z/4  \oplus	  \Z/3                                    		&	0
\\ 452 	&	113	&	0	&	1	&	1	&	0	&	27	&	  (\Z/2)^4  \oplus					  \Z/3 															&	0	&	  (\Z/2)^4  \oplus	  \Z/3                                		&	0
\\ 455 	&	455	&	2	&	0	&	0	&	0	&	39	&	  (\Z/2)^4  \oplus					  \Z/3 															&	 \Z/3                     	&	  (\Z/2)^3 \oplus \Z/4 \oplus \Z/3                  	&	0
\\ 456 	&	114	&	2	&	2	&	0	&	0	&	32	&	  (\Z/2)^3  \oplus					  \Z/3 															&	 \Z/2 \oplus (\Z/3)^3     	&	  (\Z/2)^3  \oplus	  \Z/3                                		&	1
\\ 463 	&	463	&	1	&	0	&	0	&	0	&	23	&	  \Z/2      																				&	0	&	  \Z/4                                                           			&	0
\\ 471 	&	471	&	1	&	0	&	0	&	0	&	34	&	  \Z/2  \oplus					  \Z/3 															&	0	&	  \Z/4  \oplus	  \Z/3                                    		&	0
\\ 472 	&	118	&	1	&	1	&	0	&	0	&	25	&	  (\Z/2)^3      																				&	 \Z/3                     	&	  (\Z/2)^3                                                       			&	0
\\ 479 	&	479	&	1	&	0	&	0	&	0	&	33	&	  \Z/2  \oplus					  (\Z/3)^3 															&	0	&	  \Z/4  \oplus	  (\Z/3)^3                                		&	0
\\ 483 	&	483	&	2	&	0	&	0	&	0	&	33	&	  \Z/3      																				&	  (\Z/2)^2 \oplus \Z/3    	&	  \Z/3                                                           			&	2
\\ 487 	&	487	&	1	&	0	&	0	&	0	&	24	&	  \Z/2  \oplus					  (\Z/13)^2 															&	0	&	  \Z/4  \oplus	  (\Z/13)^2                              		&	0
\\ 488 	&	122	&	0	&	0	&	1	&	0	&	28	&	  (\Z/2)^4  \oplus					  \Z/3 															&	 \Z/3                     	&	  (\Z/2)^4  \oplus	  \Z/3                                		&	0
\\ 499 	&	499	&	2	&	1	&	0	&	0	&	22	&	  (\Z/2)^2  \oplus					  (\Z/3)^2															&	0	&	  (\Z/2)^2  \oplus	  (\Z/3)^2                            		&	0
\\ 520 	&	130	&	0	&	0	&	2	&	0	&	32	&	  (\Z/2)^6      																				&	 (\Z/3)^2                 	&	  (\Z/2)^6                                                       			&	0
\\ 532 	&	133	&	3	&	0	&	0	&	0	&	33	&	  (\Z/2)^2      																				&	 \Z/2 \oplus (\Z/3)^2     	&	  (\Z/2)^2                                                       			&	1
\\ 555 	&	555	&	2	&	0	&	0	&	0	&	39	&	  \Z/3      																				&	 (\Z/2)^2 \oplus \Z/3     	&	  \Z/3                                                           			&	2
\\ 568 	&	142	&	6	&	0	&	0	&	0	&	29	&	  (\Z/2)^3  \oplus					  \Z/4 															&	 (\Z/2)^3 \oplus \Z/3     	&	  (\Z/2)^3  \oplus	  \Z/4                                		&	
\\ 595 	&	595	&	2	&	0	&	0	&	0	&	37	&	  \Z/3      																				&	 (\Z/2)^2 \oplus (\Z/3)^3 	&	  \Z/3                                                           			&	2
\\ 667 	&	667	&	1	&	0	&	0	&	0	&	32	&	0	&	 \Z/2 \oplus \Z/3         	&	0	&	1
\\ 696 	&	174	&	3	&	0	&	0	&	0	&	50	&	  \Z/2  \oplus					  \Z/3 															&	 (\Z/2)^2 \oplus (\Z/3)^2 	&	  \Z/2  \oplus	  \Z/3                                    		&	2
\\ 715 	&	715	&	2	&	0	&	0	&	0	&	43	&	0	&	 (\Z/2)^2 \oplus (\Z/3)^2 	&	0	&	2
\\ 723 	&	723	&	1	&	2	&	0	&	0	&	41	&	  (\Z/2)^2  \oplus					  \Z/3 															&	 \Z/2                     	&	  (\Z/2)^2  \oplus	  \Z/3                                		&	1
\\ 760 	&	190	&	3	&	0	&	0	&	0	&	46	&	  (\Z/2)^4      																				&	 \Z/2 \oplus (\Z/3)^2     	&	  (\Z/2)^4                                                       			&	1
\\ 763 	&	763	&	1	&	0	&	0	&	0	&	38	&	0	&	 \Z/2 \oplus \Z/3         	&	0	&	1
\\ 795 	&	795	&	2	&	0	&	0	&	0	&	55	&	  \Z/3      																				&	 (\Z/2)^2 \oplus (\Z/3)^2 	&	  \Z/3                                                           			&	2
\\ 955 	&	955	&	1	&	0	&	0	&	0	&	50	&	  (\Z/2)^4  \oplus					  (\Z/3)^2 															&	 \Z/2 \oplus \Z/3         	&	  (\Z/2)^4  \oplus	  (\Z/3)^2                            		&	1
\\ 1003 	&	1003	&	1	&	2	&	0	&	0	&	48	&	  \Z/2  \oplus					  (\Z/3)^2 															&	 \Z/3                     	&	  \Z/2  \oplus	  (\Z/3)^2                              		&	0
\\ 1027 	&	1027	&	1	&	0	&	0	&	0	&	48	&	  (\Z/2)^2  \oplus					  (\Z/3)^3 															&	 \Z/2                     	&	  (\Z/2)^2  \oplus	  (\Z/3)^3                          		&	1
\\ 1227 	&	1227	&	1	&	2	&	0	&	0	&	69	&	  \Z/4  \oplus					  \Z/8  \oplus					  \Z/3 										&	0	&	  \Z/4  \oplus  \Z/8  \oplus \Z/3        	&	
\\ 1243 	&	1243	&	1	&	2	&	0	&	0	&	58	&	  (\Z/3)^4      																				&	 \Z/2 \oplus \Z/3         	&	  (\Z/3)^4                                                     			&	1
\\ 1387 	&	1387	&	1	&	2	&	0	&	0	&	62	&	\Z/2 \oplus (\Z/3)^2 \oplus (\Z/167)^2	&	 \Z/3                     	&	\Z/2 \oplus   (\Z/3)^2  \oplus (\Z/167)^2  	&	0
\\ 1411 	&	1411	&	1	&	2	&	0	&	0	&	64	&	\Z/2 \oplus (\Z/16)^2 \oplus (\Z/43)^2	&	 \Z/3                     	&	\Z/2  \oplus   (\Z/16)^2 \oplus  (\Z/43)^2 	&	
\\ 1507 	&	1507	&	1	&	2	&	0	&	0	&	70	&	  (\Z/3)^2  \oplus					  (\Z/5)^4 															&	 \Z/2 \oplus \Z/3         	&	  (\Z/3)^2  \oplus	  (\Z/5)^4                          		&	1
\\ 1555 	&	1555	&	1	&	0	&	0	&	0	&	80	&	  (\Z/4)^8  \oplus					  (\Z/11)^2 															&	 \Z/2 \oplus \Z/3         	&	  (\Z/4)^8  \oplus	  (\Z/11)^2                        		&	
\\ \hline
\end{array}
$$

\bibliographystyle{amsplain}

\end{document}